\documentclass[11pt]{article}
\usepackage[left=1in,right=1in,top=1in,bottom=1in]{geometry}
     
\usepackage{enumitem}
\usepackage[version=4]{mhchem}
\usepackage{chemfig}    \newcommand{\cf}[1]{{\sf{#1}}}
\usepackage{tikz, graphicx}
\usepackage{multicol}
\usepackage{multirow}
\usepackage[margin=1cm]{caption}
\usepackage{subcaption}
	\usetikzlibrary{positioning}
	\usetikzlibrary{arrows}
        \tikzset{%
        fwdrxn/.style={very thick, arrows={-Stealth[length=5pt,width=5pt]}},
        revrxn/.style={very thick, arrows={-Stealth[length=5pt,width=5pt,left]}},
        revrxnN/.style={transform canvas={yshift=1.25pt}},
        revrxnS/.style={transform canvas={yshift=-1.25pt}},
        revrxnE/.style={transform canvas={xshift=1.25pt}},
        revrxnW/.style={transform canvas={xshift=-1.25pt}},
        revrxnNE/.style={transform canvas={xshift=1pt,yshift=1pt}},
        revrxnSE/.style={transform canvas={xshift=1pt,yshift=-1pt}}, 
        revrxnSW/.style={transform canvas={xshift=-1pt,yshift=-1pt}},
        revrxnNW/.style={transform canvas={xshift=-1pt,yshift=1pt}},
        }
        
        \newcommand{\ratecnst}[1]{\scriptsize\color{ratecnst}{#1}}
        \newcommand{\delayparam}[1]{{\scriptsize\color{redorange}{#1}}}
        \tikzset{%
        Snode/.style={circle, draw=black, thick,  fill=viridisgreen!30!white, fill opacity = 1, inner sep=0pt,minimum size=0.8cm, outer sep=0pt},
        Rnodelong/.style={rectangle, rounded corners, draw=black, thick, fill=viridisyellowpale, fill opacity = 1, minimum width=2.75cm, minimum height=0.7cm, outer sep=0pt},
        Rnode/.style={rectangle, rounded corners, draw=black, thick, fill=viridisyellowpale, fill opacity = 1, minimum width=1cm, minimum height=0.7cm, outer sep=0pt},
        SR/.style={very thick},
        DSR/.style={very thick, style={very thick, arrows={-Stealth[length=7.5pt,width=7.5pt]}}},
        srNW/.style={xshift=1.3pt,yshift=-1.3pt},
        srSW/.style={xshift=1.3pt,yshift=1.3pt},
        srNE/.style={xshift=-1.3pt,yshift=-1.3pt},
        srSE/.style={xshift=-1.3pt,yshift=1.3pt},
        cycopen/.style={line width=0.95cm, opacity=0.5, line cap=round}, 
        cycopenhalf/.style={line width=0.45cm,  opacity=0.5, line cap=round}, 
        }

        \tikzset{ 
            double arrow/.style args={#1 colored by #2 and #3}{
            -stealth,line width=#1,#2, 
            postaction={draw,-stealth,#3,line width=3*(#1)/4,
            shorten <=(#1)/9,shorten >=1*(#1)/4}, 
            }
        }

    \newcommand{\tikzinline}[2]
	    {\begin{center}\begin{tikzpicture}[scale={#1}]#2
	    \end{tikzpicture}\end{center}
	    \vspace{0cm}}

\usepackage{color, xcolor}

	\definecolor{orange}{RGB}{250, 140, 0}
		\newcommand\orange[1]{{\textcolor{orange}{#1}}}
	\definecolor{redorange}{RGB}{250, 87, 0}
	\definecolor{turq}{RGB}{0,153,115}
		\newcommand\turq[1]{{\textcolor{turq}{#1}}}
	\definecolor{dkturq}{RGB}{56, 128, 115}  
	    
	\definecolor{violet}{RGB}{164, 98, 234}
		
    \definecolor{graycv}{RGB}{158,158,158}
        
    \definecolor{dkgraycv}{RGB}{130, 130, 130}
    	
    \definecolor{ratecnst}{RGB}{102, 102, 102}

    \definecolor{viridisyellow}{RGB}{253,231,36}
    \definecolor{viridisyellowpale}{RGB}{239,223,81}
    \definecolor{viridisgreen}{RGB}{121,209,81}
        \definecolor{hlgreen}{RGB}{16,115,16}
    \definecolor{viridisturq}{RGB}{34,167,132}
    \definecolor{viridisblue}{RGB}{64,67,135}
    \definecolor{viridisviolet}{RGB}{68,1,84}
    \definecolor{magmayellow}{RGB}{245,218,81}
    \definecolor{magmaorange}{RGB}{239,164,80}
    \definecolor{magmapink}{RGB}{188,81,119}
    \definecolor{magmapurple}{RGB}{102,27,162}
    \definecolor{magmablue}{RGB}{4,18,129}
    \definecolor{highlightflamingo}{RGB}{252, 116, 253}
    \definecolor{highlightbabypink}{RGB}{252, 128, 165}
    \definecolor{highlightpastelpink}{RGB}{253,191,210}
    \definecolor{highlightorange}{RGB}{255, 122, 0}
    \definecolor{highlightblue}{RGB}{91, 101, 240}
    \definecolor{highlightpastelgreen}{RGB}{157, 224, 147}
    \definecolor{highlightemerald}{RGB}{31, 219, 160}
    \definecolor{highlightbabyblue}{RGB}{118, 215, 234}
    \definecolor{highlightlavender}{RGB}{230, 202, 252}

\usepackage[normalem]{ulem}
\usepackage{amsfonts, amssymb, mathtools}
\usepackage[nocompress, noadjust, sort]{cite}
	\allowdisplaybreaks	
\newcommand{\eq}[1]{\begin{align*}#1\end{align*}}
	\newcommand{\eqn}[1]{\begin{align}#1\end{align}}  
	
\newcommand{\vv}[1]{{\boldsymbol{#1}}}  
\newcommand{\mm}[1]{\mathbf{#1}}        

\newcommand\mf[1]{\mathfrak{#1}}  
\newcommand\mc[1]{\mathcal{#1}}

\newcommand{\rr}{\ensuremath{\mathbb{R}}}

\renewcommand{\epsilon}{\varepsilon}	
      
\renewcommand{\phi}{\varphi}		

\newcommand{\Id}{\mm{Id}} 		    

\DeclareMathOperator{\supp}{supp}

\newcommand{\st}{\colon}
\DeclareMathOperator{\Span}{span}

\newcommand{\kk}{\kappa}
\newcommand{\rrp}{\rr_{\geq}}
\newcommand{\rrpp}{\rr_{>}}

\newcommand{\RR}{\ensuremath{\rightleftharpoons}}
\newcommand{\FR}{\ensuremath{\rightarrow}}

\newcommand{\Gk}{\ensuremath{\mc N_{\vv \kk}}}
    \newcommand{\Gkde}{\ensuremath{\mc N_{\vv \tau,\vv \kk}}}
    \newcommand{\Gktilde}{\ensuremath{\tilde{\mc N}_{\tilde{\vv \kk}} }}
    \newcommand{\Gtilde}{\ensuremath{\tilde{\mc N}}}
\newcommand{\xx}{\vv x}
\newcommand{\yy}{\vv y}

\newcommand{\Jac}{\mm{J}}
    \newcommand{\mtxvsp}{\vphantom{\sum^{\frac{\sum^1}{2}}_\sum}} 

\newcommand{\ee}{\hat{\mm{e}}}

\usepackage{amsthm}
\usepackage{hyperref}
\usepackage{amsmath}
\usepackage{mathrsfs}
\usepackage[nameinlink]{cleveref}
    \crefname{ex}{Example}{Examples}
    \crefname{thm}{Theorem}{Theorems} 
    \crefname{lem}{Lemma}{Lemmas}
    \crefname{prop}{Proposition}{Propositions}
    \crefname{cor}{Corollary}{Corollaries} 
    \crefname{conj}{Conjecture}{Conjectures} 
    \crefname{defn}{Definition}{Definitions}
    \crefname{rmk}{Remark}{Remarks}

	\newtheorem{thm}{Theorem}[section]
	\newtheorem{lem}[thm]{Lemma}
	\newtheorem{prop}[thm]{Proposition}
	\newtheorem{cor}[thm]{Corollary}
	
	\theoremstyle{definition} 
		\newtheorem{defn}[thm]{Definition}
		\newtheorem{ex}[thm]{Example}
    	\newtheorem{rmk}[thm]{Remark}	

	\newtheoremstyle{TheoremNum}
        {\topsep}{\topsep}      
        {\itshape}              
        {}                      
        {\bfseries}             
        {.}                     
        { }                     
        {\thmname{#1}\thmnote{ \bfseries #3}}
    \theoremstyle{TheoremNum}
	
\newcommand{\df}[1]{{\bf\emph{#1}}}		

\usetikzlibrary{shapes,automata,positioning,arrows,fit} 
\usetikzlibrary{decorations.pathmorphing}

\usepackage{marginnote}

\newcommand{\emphc}[1]{\turq{#1}}
\newcommand{\SR}{DSR graph}
\newcommand{\DSR}{\mf{D}}

\definecolor{dkorange}{RGB}{247, 72, 2}

\usepackage{authblk}
\title{
    A graph-theoretic condition for delay stability of reaction systems
}
\author[1]{
        Polly Y. Yu%
}
\author[3]{
        Maya Mincheva%
}
\author[4]{
        Casian Pantea%
}
\author[1,2]{
         Gheorghe Craciun%
}
\affil[1]{\small Department of Mathematics, University of Wisconsin-Madison}
\affil[2]{\small Department of Biomolecular Chemistry, University of Wisconsin-Madison}
\affil[3]{\small Department of Mathematical Sciences, Northern Illinois University}
\affil[4]{\small Department of Mathematics, West Virginia University}
\date{} 

\begin{document}
\maketitle

\begin{abstract}
    Delay mass-action systems provide a model of chemical kinetics when past states influence the current dynamics. In this work, we provide a graph-theoretic condition for \emph{delay stability}, i.e., linear stability independent of both rate constants and delay parameters. In particular, the result applies when the system has no delay, implying asymptotic stability for the ODE system. The graph-theoretic condition is about cycles in the \emph{directed species-reaction graph} of the network, which encodes how different species in the system interact. 
\end{abstract}

\tableofcontents

\section{Introduction}
\label{sec:intro}

Mass-action kinetics is a common modelling assumption for chemical and biochemical processes. If the environment is well-mixed, the rate at which any reaction proceeds is assumed to be proportional to the concentrations of reactant species, and the ratio of reactants lost to products gained is determined by the stoichiometric coefficients. 

As with most polynomial ODEs, mass-action systems allow for complicated dynamics, which may change considerably when parameters vary. It is then noteworthy that a number of dynamical properties of a reaction system have been shown to be a function of the network structure alone, and independent of parameter values. A number of recent results relate structural features of the reaction network with the existence, uniqueness, and stability of steady states, existence of oscillations, persistence (non-extinction) of solutions, and absolute concentration robustness; for example, see~\cite{Boros2019, Feinberg1987, ConradiFeliuMinchevaWiuf2017, MinchevaRoussel2007, CraciunNazarovPantea2013_GAC, GopalkrishnanMillerShiu2014_GAC, ShinarFeinberg2010}.

The structure of a reaction network may be encoded by way of its \emph{DSR graph}, a labeled digraph related to Petri nets. First introduced to address questions of multistationary \cite{CraciunBanaji2009, CraciunFeinberg2006}, features of the DSR graph or closely related graphs have been tied to other dynamical properies, like persistence, stability, and absence of oscillation \cite{Angeli.2010aa, Angeli.2013aa, craciun2011graph}.

Under mass-action kinetics, product species are available instantaneously; however, many processes naturally involve a time delay between reactant consumption and product production, for example transmission of cellular signal~\cite{Macdonald1989}. By taking into account the influence of the past, the governing equation is a system of \emph{delay} differential equations, instead of a system of ordinary differential equations. In this work we show that the DSR graph can allow conclusions about mass-action systems \emph{with delays}.

More precisely, we show that if the DSR graph of a reaction network $\mc N$ satisfies certain conditions, then $\mc N$ is \emph{delay stable}, i.e., any positive steady states are linearly stable for all choices of parameters, incuding delay parameters. This is our  \Cref{thm:delay-SR}, which can be regarded as the main result of the paper. Moreover, this theorem also implies asymptotic stability of any positive steady states for the mass-action system \emph{without delay}, i.e., for the system of ODEs.

Here is an outline of the proof of \Cref{thm:delay-SR}: first we construct a certain network $\tilde{\mc N}$ (\emph{the modified network}) as explained in \Cref{sec:delay-modified}. Then, as an intermediary step, we relate delay stability of $\mc N$ with the DSR graph of $\tilde{\mc N}$ in  \Cref{sec:delay-SRintro}. Finally, the relation between the DSR graphs of $\mc N$ and $\tilde{\mc N}$ is discussed in \Cref{sec:delay-SR}, and we arrive at a sufficient condition for delay stability involving the DSR graph of the initial network $\mc N$. A more detailed summary of the proof of \Cref{thm:delay-SR} is given as a diagram of implications in \Cref{fig:implications}, which  can also be viewed as a roadmap of this paper. Additional details are provided in \Cref{sec:implications}. 
Background on delay mass-action systems and the DSR graph is given in \Cref{sec:def,sec:delay-SRintro}. While examples are discussed throughout the text, \Cref{ex:DNA-SR,ex:CST} have biological motivation.

Although this entire work is discussed in terms of mass-action kinetics, the results hold for more general kinetics under mild conditions. Mass-action kinetics assumes that the reaction rate function $\mathscr{K}_{\yy\to\yy'}$ is proportional to the concentrations of reactants (with exponents given by stoichiometric coefficients). The main result in this paper holds when the rate function $\mathscr{K}_{\yy\to\yy'}$ satisfies $\frac{\partial \mathscr{K}_{\yy\to\yy'}}{\partial x_j} \geq 0$ for all $j$, and $\frac{\partial \mathscr{K}_{\yy\to\yy'}}{\partial x_i} > 0$ for any $i \in \supp(\yy)$. This generalization is permitted because of \cite[Remark 4.3]{CraciunMinchevaPanteaYu2020} and \cite{CraciunBanaji2009}.

\section{Mass-action systems with delay}
\label{sec:def}

In this section, we introduce mass-action systems with delay and some related notions, previously defined in~\cite{CraciunMinchevaPanteaYu2020}. For more details about delay mass-action systems, see~\cite{Roussel1996, CraciunMinchevaPanteaYu2020}. Let $\rrp^n$ denote the set of vectors with non-negative coordinates, and let $\rrpp^n$ denote the set of vectors with positive coordinates. The support $\supp(\yy)$ of a vector $\yy \in \rr^n$ is the set of indices for which $y_i \neq 0$. The cardinality of a set $X$ is denoted $|X|$. Finally, given $\xx \in \rrpp^n$ and $\yy \in \rr^n$, define vector exponentiation as $\xx^{\yy} = x_1^{y_1}x_2^{y_2}\cdots x_n^{y_n}$.

\begin{defn}
\label{def:rn}
    A \df{chemical reaction network} $\mc N = (\mc V, \mc R)$, or \df{reaction network}, is a finite directed graph, where each vertex $\yy \in \mc V$, called a \df{complex}, is a vector in $\rrp^n$. Each edge $(\yy, \yy') \in \mc R$, called a \df{reaction}, is denoted $\yy \to \yy'$. 
\end{defn}

\begin{rmk}
The definition above is equivalent to the classical definition of a reaction network being a triple $(\mc S, \mc C, \mc R)$, where $\mc S$ is the set of \emph{species}, $\mc C$ is the set of \emph{complexes} and $\mc R$ is the set of \emph{reactions}~\cite{CraciunBanaji2010, CraciunFeinberg2005, CraciunFeinberg2006,  CraciunYu2018_review, Feinberg_lecture}. Indeed, given $\mathcal{N}$ as above, the set of \df{species} is identified (by an abuse of notation) to the standard basis $\{\ee_1, \ee_2,\ldots, \ee_n\}$ of $\rr^n$, and the complexes are non-negative linear combinations of the species. Conversely, given a triple $(\mc S, \mc C, \mc R)$ as described in \cite{CraciunFeinberg2005, Feinberg_lecture}, we can use the same identification between the set of species and the standard basis of $\rr^n$ to write a complex as a vector $\vv y \in \rr_{\geq 0}^n$. If $i \in \supp (\yy)$, then we say that $\cf{X}_i$ (identified with $\ee_i$) is a species in the complex $\vv y$. 
\end{rmk} 

For any reaction $\vv y \to\vv  y'$, we call $\vv y$ a \df{reactant complex}, and  $\vv y'$ a \df{product complex}. A species in $\vv y$ is a \df{reactant species} of the reaction $\vv y \to\vv  y'$, and a species in $\vv y'$ is a \df{product species} of the reaction. In other words, $\supp(\vv y)$ consists of the reactant species, while $\supp(\vv y')$ consists of the product species. A \df{bispecies reaction} is one with two reactant species. To avoid excessive subscripts later, we enumerate the set of reactions: $\mc R = \{ \yy_r \to \yy'_r \st 1 \leq r \leq R\}$. When indexing over the set of reactions, as in $r \in \mc R$, we mean the reaction $\yy_r \to \yy'_r$.

To model how concentrations of species change over time, we make an assumption on the rate at which each reaction proceeds. One of the most common models in chemistry and biochemistry is \emph{mass-action kinetics}, which assumes that the reaction rate is proportional to the concentration of reactant species. 

\begin{defn}\label{def:mas}
A \df{mass-action system} $\Gk$ is a reaction network $\mc N = (\mc V, \mc R)$ together with a vector of rate constants $\vv \kk \in \rrpp^{\mc R}$.  The dynamics of the concentration vector $\xx(t)$ is given by
	\eqn{\label{eq:mas}
		\dot{\xx}(t) &= \sum_{r \in \mc R} \kk_r [\xx(t)]^{\yy_r} (\yy'_r-\yy_r).
    }
\end{defn}

In a mass-action system, the consumption of reactants and the production of products occur simultaneously. However, some systems have an inherent delay between consumption and production~\cite{Roussel1996}; for example, the binding of two single-stranded DNA molecules via nucleation-propagation mechanism, or an enzyme needing to change its conformation after taking part in a reaction. In these scenarios, the reactants are consumed immediately while the products become available at a later time. We define a delay mass-action system to be a mass-action system augmented with a vector of delay parameters, one non-negative parameters for each reaction.

\begin{defn}\label{def:dmas}
    A \df{delay mass-action system} $\Gkde$ is a mass-action system $\Gk$ with a vector of delays $\vv\tau \in \rrp^{\mathcal R}$. The dynamics of the concentration vector $\xx(t)$ is given by 
    \eqn{\label{eq:dmas}
		\dot{\xx}(t) &= \sum_{r \in \mc R} \kk_r [\xx(t- \tau_r)]^{\yy_r} \yy'_r - \sum_{r \in \mc R}
		 \kk_r [\xx(t)]^{\yy_r} \yy_r.
	}
\end{defn}

For an initial value problem of \eqref{eq:dmas}, the initial data is a function defined on the interval $[-\bar{\tau}, 0]$ where $\bar{\tau} = \max_r \tau_r$. If all reactions occur without delay, i.e., $\vv \tau=\vv 0$, then from the perspective of dynamics, $\Gkde$ is not different from $\Gk$~\cite{BanksRobbinsSutton2013}. Indeed, the ODE system \eqref{eq:mas} is identical to the delay system  \eqref{eq:dmas} when $\vv\tau = \vv 0$. Moreover, the ODE system \eqref{eq:mas} has only non-negative solutions if the initial condition is non-negative~\cite{LiptakHangosPitukSzederkenyi2018}; similarly, the first quadrant is also forward invariant for the delay system~\cite{Bodner2000}. 

The systems \eqref{eq:mas} and \eqref{eq:dmas} share the same set of positive steady states~\cite{LiptakHangosPitukSzederkenyi2018}; in other words, a positive constant solution $\vv x(t) \equiv \vv x^*$ is a steady state for the delay system \eqref{eq:dmas} if and only if it is a steady state of the ODE system \eqref{eq:mas}. A positive steady state $\vv x^*$ is also called an \df{equilibrium}. Note, however, that although the sets of equilibria are the same for the ODE and delay systems, they generally satisfy \emph{different} conservation relations; solving for an equilibrium with a particular initial condition for the delay system can be difficult. A mass-action system with positive initial data $\vv \theta \in \rrpp^n$ may have a conservation relation  
    \eq{ 
        \xx(t) - \vv\theta \in S,
    }
where $S = \Span\{ \yy'_r-\yy_r \st r \in \mc R\}$ is the stoichiometric subspace. A delay mass-action system with positive initial data $\vv \theta$ defined on the interval $[-\bar{\tau}, 0]$ may admit a conservation relation~\cite{LiptakHangosPitukSzederkenyi2018} 
    \eq{ 
        \xx(t) - \vv \theta(0) 
        + \sum_{r \in \mc R} \kk_r  \left(  \int_{t-\tau_r}^t  [\xx(s)]^{\yy_r}  \, ds - \int_{-\tau_r}^0 [\vv \theta(s)]^{\yy_r}   \, ds \right) \yy_r 
        \in S.
    }
In this paper, we only consider systems whose stoichiometric subspace $S$ is all of $\rr^n$, and so the ODE and delay systems share the same set of equilibria regardless of initial condition.

\begin{ex}
\label{ex:DNA}
We now illustrate how we represent a delay mass-action system $\Gkde$ using a simple model for DNA duplex formation. A DNA helix comprises of two complementary single-stranded DNA. During the nucleation step, several base pairs must find their partners in the complementary strand. However, once that happens, the two strands zip together like a zipper; this is the propagation step~\cite[Chapter 23]{NuclProp-DNA_book}.

As a toy model of duplex formation, consider two single-stranded DNA ($\cf{S}$) forming a duplex ($\cf{D}$) reversibly with some time delays. Reality is of course much more complicated; we are neglecting that the single-stranded DNA should be complementary, not identical. Moreover, in DNA replication, usually one strand forms a template, and the other strand is built from individual nucleotides. Finally, we are also neglecting the physical process whereby the double-stranded DNA twists to form a helix, and the thermodynamics when long sequences are involved. In this toy model, we assume that the delay parameters are proportional to the length of the DNA sequence. Moreover, we include the degradation of $\cf{D}$, and the synthesis and degradation of $\cf{S}$. 

The delay mass-action system is 
    \tikzinline{1}{
        \node (2S) at (0,0) [left] {$\cf{2S}$};
        \node (D) at (1.5,0) [right] {$\cf{D}$};
        \node (0) at (3.55,0) [right] {$\cf{0}$};
        \node (S) at (5.6,0) [right] {$\cf{S}$};
        \draw [revrxn, revrxnN] (2S) -- (D) node [above, midway] {\ratecnst{$\kk_1$}, \delayparam{$\tau_1$}};
        \draw [revrxn, revrxnS] (D) -- (2S) node [below, midway] {\ratecnst{$\kk_2$}, \delayparam{$\tau_2$}};
        \draw [fwdrxn] (D) -- (0) node [above, midway] {\ratecnst{$\kk_3$}};
        \draw [revrxn, revrxnN] (0) -- (S) node [above, midway] {\ratecnst{$\kk_{4}$}};
        \draw [revrxn, revrxnS] (S) -- (0) node [below, midway] {\ratecnst{$\kk_{5}$}};
    }
where the rate constants and delay parameters (if non-zero) are shown as edge labels. By an abuse of notation, let $\cf{S}$ and $\cf{D}$ denote the concentration variables of the single-stranded and double-stranded DNA respectively. The associated system of delay differential equations 
    \eq{ 
        \dot{\cf{S}}(t) &= \kk_{4} - \kk_{5} \cf{S}(t) - 2 \kk_1 [\cf{S}(t)]^2 + 2\kk_2 \cf{D}(t-\tau_2) \\
        \dot{\cf{D}}(t) &= - \kk_3 \cf{D}(t) + \kk_1 [\cf{S}(t-\tau_1)]^2 - \kk_2 \cf{D}(t)
    }
has a single equilibrium for any choice of positive rate constants. This is given by 
    \eq{ 
        \frac{\kk_1 \cf{S}^2}{\kk_2 + \kk_3}
        = \cf{D} = 
        \frac{ 2\kk_1 \cf{S}^2 + \kk_{5}\cf{S} - \kk_4}{2\kk_2}. 
    }
The resulting quadratic equation $2\kk_1\kk_3 \cf{S}^2 + \kk_5(\kk_2+\kk_3) \cf{S} - \kk_{4}(\kk_2+\kk_3) = 0$ always has a positive root and a negative root, leading to a unique positive equilibrium.
\end{ex}

One way to study the asymptotic stability of a delay system's equilibrium $\xx^*$ is by linearizing about $\xx^*$. The linearized system has a solution of the form $\xx(t) = \xx^* + \vv a e^{\lambda t}$ with $\vv a \neq \vv 0$ if and only if 
    \eqn{\label{eq:char}  
        \det( \Jac_\lambda(\xx^*, \vv\kk, \vv\tau) - \lambda \Id) = 0,
    }
where
    \eqn{\label{eq:Jlambda} 
        \Jac_\lambda(\xx,\vv\kk, \vv\tau) 
        = \sum_{r \in \mc R} \kk_r
        \left( \mtxvsp
             \frac{\partial \xx^{\yy_r}}{\partial x_1}  \left( e^{-\lambda \tau_r} \yy'_r - \yy_r\right) 
             \,\,,\,\,\,\cdots \,\,\,,\,\, 
            \frac{\partial \xx^{\yy_r}}{\partial x_n}  \left( e^{-\lambda \tau_r} \yy'_r - \yy_r\right) 
        \right) .
    }
A complete derivation the \df{characteristic equation} \eqref{eq:char} can be found in \cite{CraciunMinchevaPanteaYu2020}. Note that if $\vv\tau = \vv 0$, then $\Jac_\lambda(\xx, \vv\kk,\vv 0)$ is the Jacobian matrix $\Jac(\xx,\vv\kk)$ of the corresponding ODE system \eqref{eq:mas}.

\begin{ex}
\label{ex:DNA-char}
We return to \Cref{ex:DNA}. The matrices 
    \eq{ 
        \Jac_\lambda(\xx^*,\vv\kk,\vv\tau) = 
        \begin{pmatrix}
        -\kk_5 - 4 \kk_1 \cf{S}^* & 2\kk_2 e^{-\lambda \tau_2} \\
        2\kk_1 \cf{S}^* e^{-\lambda \tau_1} & -\kk_2 -\kk_3  
        \end{pmatrix}
        \quad \text{and}\quad 
        \Jac (\xx^*, \vv\kk) = 
        \begin{pmatrix}
        -\kk_5 - 4 \kk_1 \cf{S}^* & 2\kk_2  \\
        2\kk_1 \cf{S}^*  & -\kk_2 -\kk_3  
        \end{pmatrix}
    }
are clearly related. The characteristic equation \eqref{eq:char} of the delay system is 
    \eq{ 
        0 = \lambda^2 + \lambda \left( 4\kk_1 \cf{S}^*  + \kk_2 + \kk_3 + \kk_5\right) 
         + \left(4\kk_1\cf{S}^* + \kk_5  \right)\left(\kk_2 + \kk_3 \right)
        - 4 \kk_1 \kk_2 \cf{S}^* e^{-\lambda(\tau_1+\tau_2)}.
    }
Note that the characteristic equation is generally a polynomial of $\lambda$ and $e^{-\lambda \tau_r}$, with coefficients that depend on $\vv\kk$ and $\xx^*$. 
\end{ex}

If every root $\lambda$ of the characteristic equation has negative real part, then the equilibrium $\xx^*$ is asymptotically stable. In this work, we are interested in two stronger notions of stability: one independent of the choice of non-negative delay parameters, and another independent of both the choice of positive rate constants and non-negative delay parameters.

\begin{defn}
\label{def:absstab}
    A mass-action system $\Gk$ is \df{absolutely stable} if for any equilibrium and any choice of delay parameters $\vv\tau \geq 0$, every root of the characteristic equation \eqref{eq:char} of the delay mass-action system $\Gkde$ has negative real part.
\end{defn}

\begin{defn}
\label{def:delaystab}
    A reaction network $\mc N$ is \df{delay stable} if the delay mass-action system $\Gkde$ is absolutely stable for any choices of $\vv \kk > \vv 0$.
\end{defn}

In \cite{CraciunMinchevaPanteaYu2020}, we provided an algebraic condition for delay stability for a class of reaction networks called \df{non-autocatalytic networks}. Later in this work, we will work with the slightly more restrictive networks with no \df{one-step catalysis}.

\begin{defn}\label{def:NAC}
    A reaction $\yy \to \yy'$ is said to be an \df{autocatalytic reaction} if $\supp(\yy) \cap \supp(\yy') \neq \emptyset$, and $y'_i > y_i$ for some $i \in \supp(\yy) \cap \supp(\yy')$. A reaction network is a \df{non-autocatalytic network} if it has no autocatalytic reaction. 
\end{defn} 
\begin{defn}\label{def:N1C}
    A reaction $\yy \to \yy'$ is said to be a \df{one-step catalysis} if $\supp(\yy) \cap \supp(\yy') \neq \emptyset$.
\end{defn}
An autocatalytic reaction is one where a species promotes its own growth. In contrast, a one-step catalysis is a reaction that involves a reactant species that is not totally consumed. Note that an autocatalytic reaction is a one-step catalysis. Therefore, a reaction network with \emph{no} one-step catalysis is a non-autocatalytic network.

The algebraic condition in \cite{CraciunMinchevaPanteaYu2020}, based on \cite{HofbauerSo2000}, involves yet another matrix, the \df{modified Jacobian matrix} of a reaction network:
    \eqn{\label{eq:Jtilde} 
        \tilde{\Jac} (\xx, \vv\kk) 
        = \sum_{r\in \mc R} \kk_r 
        \left( \mtxvsp
             \frac{\partial \xx^{\yy_r}}{\partial x_1}  \left(  \yy'_r + \tilde{\yy}^{(1)}_r\right) 
             \,\,,\,\,\,\cdots \,\,\,,\,\, 
            \frac{\partial \xx^{\yy_r}}{\partial x_n}  \left(  \yy'_r + \tilde{\yy}_r^{(n)}\right) 
        \right), 
    }
where $\tilde{\yy}^{(j)} = (y_1,\ldots, -y_j, \ldots, y_n)^\top$. This matrix is reminiscent of the Jacobian matrix $\Jac$; in the $j$th column, we flip the sign of the reactant complex off-diagonal. Finally, recall that a matrix $\mm A$ is a $P_0$-matrix if it has only non-negative principal minors~\cite{FiedlerPtak1962, Johnson1974}.

\begin{prop}[Corollary 4.5,~\cite{CraciunMinchevaPanteaYu2020}]
\label{cor:delay-alg}
Let $\mathcal N$ be a non-autocatalytic network. Let $\Jac$ and $\tilde{\Jac}$ be the Jacobian and modified Jacobian matrices, viewed as functions of $\vv x > \vv 0$ and $\vv \kk > \vv 0$. Suppose $\det \Jac \neq 0$, $\tilde{\Jac}_{ii}<0$ for all $i$, and  $-\tilde{\Jac}$  is a $P_0$-matrix for all choices of  $\vv \kk > \vv 0$  and all equilibria $\vv x > \vv 0$  of $\Gk$. Then $\mathcal{N}$ is delay stable. 
\end{prop}

\begin{ex}
Consider the delay system $\Gkde$
    \tikzinline{1}{
        \node (xy) at (-1,-1) [left] {$\cf{X+Y}$};
        \node (z) at (0.5,-1) [right] {$\cf{Z}$};
        \node (x) at (3,-1) [left] {$\cf{X}$};
        \node (y) at (4.5,-1) [right] {$\cf{Y}$};
        \node (0a) at (-3,0) [left] {$\cf{X}$};
        \node (0b) at (-1.5,0) [right] {$\cf{0}$};
        \node (1a) at (1,0) [left] {$\cf{Y}$};
        \node (1b) at (2.5,0) [right] {$\cf{0}$};
        \node (2a) at (5,0) [left] {$\cf{Z}$};
        \node (2b) at (6.5,0) [right] {$\cf{0}$};
        \draw [fwdrxn] (xy) -- (z) node [above, midway] {\ratecnst{$\kk_5$}, \delayparam{$\tau_1$}};
        \draw [fwdrxn] (x) -- (y) node [above, midway] {\ratecnst{$\kk_6$}, \delayparam{$\tau_2$}};
        \draw [revrxn, revrxnN] (0a) -- (0b) node [above, midway] {\ratecnst{$\kk_1$}};
        \draw [revrxn, revrxnS] (0b) -- (0a) node [below, midway] {\ratecnst{$\kk_{4}$}};
        \draw [fwdrxn] (1a) -- (1b) node [above, midway] {\ratecnst{$\kk_2$}};
        \draw [fwdrxn] (2a) -- (2b) node [above, midway] {\ratecnst{$\kk_3$}};
    }
which evolves according to
    \eq{ 
        \dot{x}(t) &= \kk_4 - \kk_1 x(t)
            - \kk_5 x(t) y(t)
            - \kk_6 x(t)
        \\
        \dot{y}(t) &= \hphantom{\kk_4} {}-\kk_2 y(t)
            - \kk_5 x(t) y(t)
            + \kk_6 x(t-\tau_2)
        \\
        \dot{z}(t) &= \hphantom{\kk_4} {}-\kk_3 z(t)
            + \kk_5 x(t-\tau_1) y(t-\tau_1).
    }
The Jacobian (of the ODE system) 
    \eq{ 
    \det \Jac &=
    \det \begin{pmatrix}
        -\kk_1-\kk_5 y -\kk_6 & -\kk_5 x & 0 \\
        -\kk_5 y + \kk_6 & -\kk_2-\kk_5 x & 0 \\
        \kk_5 y & \kk_5 x & -\kk_3
    \end{pmatrix}
    \\&= -\kk_3 \left( 
         \kk_1\kk_2 + \kk_1\kk_5x + \kk_2\kk_5 y + \kk_2\kk_6 
         + 2\kk_5\kk_6 x 
    \right) 
    }
is non-zero. The modified Jacobian matrix is
    \eq{ 
        \tilde{\Jac} &=
        \begin{pmatrix}
        -\kk_1-\kk_5 y -\kk_6 & \emphc{+}\kk_5 x & 0 \\
        \emphc{+}\kk_5 y + \kk_6 & -\kk_2-\kk_5 x & 0 \\
        \kk_5 y & \kk_5 x & -\kk_3
    \end{pmatrix}.
    }
Clearly, the diagonal terms are all negative. The $2\times 2$ minors of $-\tilde{\Jac}$ are 
    \eq{ 
    [\tilde{\Jac}]_{1,2} &= \kk_1\kk_2 + \kk_1 \kk_5 x + \kk_2\kk_5 y + \kk_2\kk_6 ,
    \\
    [\tilde{\Jac}]_{1,3} &= \kk_3( \kk_1 + \kk_5 y + \kk_6), 
    \\
    [\tilde{\Jac}]_{2,3} &= \kk_3(\kk_2 + \kk_5 x) ,
    }
while $\det(- \tilde{\Jac}) >0$. So $\tilde{-\Jac}$ is a $P_0$-matrix for any $\xx > \vv 0$ and $\vv\kk > \vv 0$. In particular, any equilibrium of the delay system $\Gkde$ is asymptotically stable, independent the choice of delay parameters $\vv\tau \geq \vv 0$ and the choice of rate constants $\vv\kk > \vv 0$. In other words, $\mc N$ is delay stable. 
\end{ex}

\section{Modified Jacobian and its corresponding reaction network}
\label{sec:delay-modified}

\Cref{cor:delay-alg} converted a problem involving a transcendental equation, the characteristic equation, to a purely algebraic one, involving determinants of matrices. One may ask what is the significance of the modified Jacobian matrix $\tilde{\Jac}$, or perhaps the biological interpretation of the result. The remainder of this paper attempts to address the second point; we consider the directed species-reaction (DSR) graph of a reaction network, and show that a lack of certain types of cycles in the DSR graph implies the algebraic condition in \Cref{cor:delay-alg}, thus avoiding the modified Jacobian matrix. Before reaching that point in \Cref{thm:delay-SR}, first we demonstrate that the modified Jacobian matrix $\tilde{\Jac}$ is in a sense the Jacobian matrix of a different reaction network.

In this section, we describe how to construct a mass-action system whose Jacobian matrix is the modified Jacobian matrix \eqref{eq:Jtilde}. For the purpose of communication, we refer to the starting network as the \lq\lq original network\rq\rq\ and the newly constructed network as the \lq\lq modified network\rq\rq. We introduce the procedure via an example. 

\begin{ex}
\label{ex:modifiedNetwork}
The mass-action system $\mc N_{\vv\kk}$
    \tikzinline{1}{
        \node (xy) at (0,0) [left] {$\cf{X}+\cf{Y}$};
        \node (z) at (1.5,0) [right] {$\cf{2Z}$};
        \draw [revrxn, revrxnN] (xy) -- (z)  node [midway, above] {\ratecnst{$\kk_1$}};
        \draw [revrxn, revrxnS] (z)--(xy) node [midway, below] {\ratecnst{$\kk_2$}};
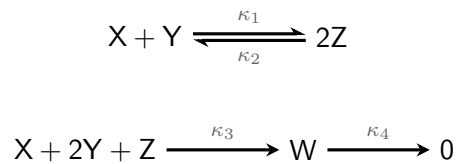
        \begin{scope}[transform canvas={xshift=-10pt}]
        \node (xyz) at (0,-1.5) [left] {$\cf{X}+\cf{2Y}+\cf{Z}$};
        \node (w) at (1.5, -1.5) [right] {$\cf{W}$};
        \node (0) at (3.5,-1.5) [right] {$\cf{0}$};
        \draw [fwdrxn] (xyz) -- (w) node [midway, above] {\ratecnst{$\kk_3$}};
        \draw [fwdrxn] (w) -- (0) node [midway, above] {\ratecnst{$\kk_4$}};
        \end{scope}\node at (0,-1.5) {};
    }
has  
    \eq{ 
        \tilde{\Jac}(\xx,\vv\kk)
        &= \begin{pmatrix} 
            - \kk_1y -\kk_3y^2z & 
            \kk_1 x + 2\kk_3 xyz & 
            2\kk_2z + \kk_3xy^2   & 0 \\
            \kk_1 y + 2\kk_3y^2z  & 
            - \kk_1 x -4\kk_3xyz  & 
            2\kk_2z + 2\kk_3xy^2 & 0 \\
            2\kk_1 y + \kk_3y^2z & 
            2\kk_1 x + 2\kk_3 xyz  & 
            -4\kk_2z -\kk_3 xy^2   & 0 \\
            \kk_3y^2z & 
            2\kk_3 xyz & 
            \kk_3xy^2 & 
            -\kk_4 
        \end{pmatrix}
    } 
as its modified Jacobian matrix. In constructing our modified network, we do not change the reactions involving less than two reactant species. For any reaction involving two or more reactant species, we create new reactions where the reactant species are migrated to the product side; for example, the reaction $\cf{X}+\cf{Y}\rightarrow \cf{2Z}$ splits into two reactions: $\cf{X}\rightarrow\cf{2Z}+\cf{Y}$ and $\cf{Y}\rightarrow\cf{2Z}+\cf{X}$. 
    
Fix a positive state $\xx^* = (x^*, y^*, z^*, w^*)^\top$. The modified mass-action system $\Gktilde$ contains the reactions
    \tikzinline{1}
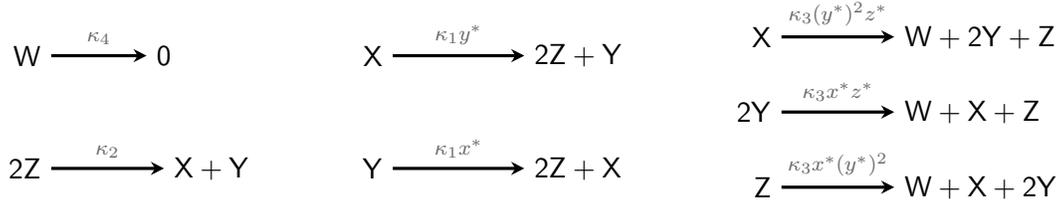
{\begin{scope}[transform canvas={xshift=-10pt}]
        \begin{scope}[transform canvas={xshift=70pt}]
        \node (w) at (-8.5,0) [left] {$\cf{W}$};
        \node (0) at (-7,0) {$\cf{0}$};
        \draw [fwdrxn] (w) -- (0) node [midway, above] {\ratecnst{$\kk_4$}};
        \node (2z) at (-8.5,-1.5) [left] {$\cf{2Z}$};
        \node (xyz) at (-7,-1.5) [right] {$\cf{X}+\cf{Y}$};
        \draw [fwdrxn] (2z) -- (xyz) node [midway, above] {\ratecnst{$\kk_2$}};
        \end{scope}
        \begin{scope}[transform canvas={xshift=0pt}]
        \node (2) at (-1.5,0) [left] {$\cf{X}$};
        \node (2t) at (0.25,0) [right] {$\cf{2Z}+\cf{Y}$};
        \draw [fwdrxn] (2) -- (2t) node [midway, above] {\ratecnst{$\kk_1y^*$}};
        \node (2) at (-1.5,-1.5) [left] {$\cf{Y}$};
        \node (2t) at (0.25,-1.5) [right] {$\cf{2Z}+\cf{X}$};
        \draw [fwdrxn] (2) -- (2t) node [midway, above] {\ratecnst{$\kk_1x^*$}};
        \end{scope}
        \begin{scope}[transform canvas={xshift=-45pt}]
        \node (2) at (5.25,0.25) [left] {$\cf{X}$};
        \node (2t) at (6.75,0.25) [right] {$\cf{W}+\cf{2Y}+\cf{Z}$};
        \draw [fwdrxn] (2) -- (2t) node [midway, above] {\ratecnst{$\kk_3(y^*)^2z^*$}};
        \node (2) at (5.25,-0.75) [left] {$\cf{2Y}$};
        \node (2t) at (6.75,-0.75) [right] {$\cf{W}+\cf{X}+\cf{Z}$};
        \draw [fwdrxn] (2) -- (2t) node [midway, above] {\ratecnst{$\kk_3x^*z^*$}};
        \node (2) at (5.25,-1.75) [left] {$\cf{Z}$};
        \node (2t) at (6.75,-1.75) [right] {$\cf{W}+\cf{X}+\cf{2Y}$};
        \draw [fwdrxn] (2) -- (2t) node [midway, above] {\ratecnst{$\kk_3x^*(y^*)^2$}};
        \end{scope}
    \end{scope}
    \node at (6.8,-1.7) {};
    \node at (6.8,0.6) {};
    \node at (-6.8,0) {};
    }
and whose Jacobian matrix is
    \eq{ 
        {\Jac}(\emphc{\xx}; \tilde{\vv\kk}(\vv\kk, \xx^*))
        &= \begin{pmatrix} 
            - \kk_1y^* -\kk_3(y^*)^2z^* & 
            \kk_1 x^* + 2\kk_3 x^*z^*\emphc{y} & 
            2\kk_2\emphc{z} + \kk_3x^*(y^*)^2   & 0 \\
            \kk_1 y^* + 2\kk_3(y^*)^2z^*  & 
            - \kk_1 x^* -4\kk_3x^*z^*\emphc{y}  & 
            2\kk_2\emphc{z} + 2\kk_3x^*(y^*)^2 & 0 \\
            2\kk_1 y^* + \kk_3(y^*)^2z^* & 
            2\kk_1 x^* + 2\kk_3 x^*z^*\emphc{y} & 
            -4\kk_2\emphc{z} -\kk_3 x^*(y^*)^2   & 0 \\
            \kk_3(y^*)^2z^* & 
            2\kk_3 x^*z^*\emphc{y} & 
            \kk_3x^*(y^*)^2 & 
            -\kk_4 
        \end{pmatrix}, 
    } 
    which depends on $\xx^* = (x^*, y^*, z^*, w^*)^\top$ as well as $\emphc{\xx} = (x, \emphc{y}, \emphc{z}, w)^\top$. In particular, $\tilde{\Jac}(\xx^*, \vv\kk) = \Jac(\xx^*; \tilde{\vv\kk}(\vv\kk, \xx^*))$. 
\end{ex}

It is worth emphasizing that the modified Jacobian matrix $\tilde{\Jac}(\xx,\vv\kk)$ happens to be the Jacobian matrix of another mass-action system, with carefully chosen rate constants and at a specific state. The two mass-action systems do not necessarily share the same set of equilibria. There are generally more reactions in the modified network, and more importantly, its rate constants depend on a chosen state for the original system.

For the general procedure to construct the modified network, consider a mass-action system $\Gk$ consisting of a single reaction:
    \tikzinline{1}{
        \node (r) at (0,0) [left] {$a_1{\cf{X}_1} + a_2{\cf{X}_2} + a_3{\cf{X}_3}$};
        \node (p) at (1.5,0) [right] {$b_1{\cf{X}_1} + b_2{\cf{X}_2} + b_3{\cf{X}_3} + b_4{\cf{X}_4} $};
        \draw [fwdrxn] (r) -- (p) node [above, midway] {\ratecnst{$\kk$}};
    }
with $a_1$, $a_2$, $a_3 > 0$ and $b_1,\ldots, b_4\geq 0$. Fix a positive state $\xx^*$. Define the modified mass-action system $\Gktilde$ with the reactions 
    \tikzinline{1}{
        \node (r1) at (0,0) [left] {$a_1{\cf{X}_1}$};
        \node (p1) at (1.5,0) [right] {$b_1{\cf{X}_1} + (a_2+b_2){\cf{X}_2} + (a_3+b_3){\cf{X}_3} + b_4{\cf{X}_4}$};
        \draw [fwdrxn] (r1) -- (p1) node [above, midway] {\footnotesize $\tilde{\kk}_1$};
        \node (r2) at (0,-1) [left] {$a_2{\cf{X}_2}$};
        \node (p2) at (1.5,-1) [right] {$(a_1+b_1){\cf{X}_1} + b_2{\cf{X}_2} + (a_3+b_3){\cf{X}_3} + b_4{\cf{X}_4}$};
        \draw [fwdrxn] (r2) -- (p2) node [above, midway] {\footnotesize $\tilde{\kk}_2$};
        \node (r3) at (0,-2) [left] {$ a_3{\cf{X}_3}$};
        \node (p3) at (1.5,-2) [right] {$(a_1+b_1){\cf{X}_1} + (a_2+b_2){\cf{X}_2} + b_3{\cf{X}_3} + b_4{\cf{X}_4}$};
        \draw [fwdrxn] (r3) -- (p3) node [above, midway] {\footnotesize $\tilde{\kk}_3$};
    }
where $\tilde{\kk}_1 = \kk (x_2^*)^{a_2}(x_3^*)^{a_3}$, $\tilde{\kk}_2 = \kk (x_1^*)^{a_1}(x_3^*)^{a_3}$ and $\tilde{\kk}_3 = \kk (x_1^*)^{a_1} (x_2^*)^{a_2}$. 

The matrix $\Jac(\xx; \tilde{\vv\kk}(\vv\kk, \xx^*))$ has contributions from three reactions, each filling a column:
    \eq{
        \Jac(\xx; \tilde{\vv\kk}(\vv\kk, \xx^*))
        =
        \begin{pmatrix} 
            \tilde{\kk}_1 \frac{\partial x_1^{a_1}}{\partial x_1} (b_1 - a_1) & 
            \,\,\tilde{\kk}_2 \frac{\partial x_2^{a_2}}{\partial x_2}  (b_1 + a_1) \,\, & 
            \tilde{\kk}_3 \frac{\partial x_3^{a_3}}{\partial x_3} (b_1 + a_1) & \quad 0\quad 
            \\[6pt]
            \tilde{\kk}_1 \frac{\partial x_1^{a_1}}{\partial x_1} (b_2 + a_2) & 
            \tilde{\kk}_2 \frac{\partial x_2^{a_2}}{\partial x_2}  (b_2 - a_2) & 
            \tilde{\kk}_3 \frac{\partial x_3^{a_3}}{\partial x_3} (b_2 + a_2) & 0
            \\[6pt]
            \tilde{\kk}_1 \frac{\partial x_1^{a_1}}{\partial x_1} (b_3 + a_3) & 
            \tilde{\kk}_2 \frac{\partial x_2^{a_2}}{\partial x_2} (b_3 + a_3) & 
            \tilde{\kk}_3 \frac{\partial x_3^{a_3}}{\partial x_3} (b_3 - a_3) & 0 
            \\[6pt]
            \tilde{\kk}_1 \frac{\partial x_1^{a_1}}{\partial x_1} (b_4) & 
            \tilde{\kk}_2 \frac{\partial x_2^{a_2}}{\partial x_2} (b_4) & 
            \tilde{\kk}_3 \frac{\partial x_3^{a_3}}{\partial x_3} (b_4) & 0
        \end{pmatrix}.
    }
In each column, there is a sign change off-diagonal, just as one expects in the modified Jacobian matrix $\tilde{\Jac}$ of the network $\mc N$. Moreover, the coefficient, say in first column, is 
    \eq{ 
        \tilde{\kk}_1 \frac{\partial x^{a_1}}{\partial x_1} = \kk a_1 x_1^{a_1-1} (x_2^*)^{a_2}(x_3^*)^{a_3}.
    }
A similar calculation can be done for each column. Thus, when $\xx = \xx^*$, the modified Jacobian matrix $\tilde{\Jac}(\xx^*, \vv\kk)$ is exactly $\Jac(\xx^*; \tilde{\vv\kk}(\vv\kk, \xx^*))$.

The construction above generalizes to reactions involving even more species. What is remarkable is that the resulting modified network does not depend on the choice of rate constants and the positive state. We now formally state the construction for a general reaction network. Let $\ee_i$ be the $i$th standard basis vector of $\rr^n$. 

\begin{defn}\label{def:modifiedMAS}
    Let $\Gk$ be a mass-action system  and fix $\xx^* \in \rrpp^n$. Let $\tilde{\mc R}$ be the set of the following reactions, with rate constants $\tilde{\vv\kk}$ as specified:
    \begin{enumerate}[label={\textup{(\alph*)}}]
        \item include any reaction $\yy \to \yy' \in \mc R$ with $|\supp(\yy)| \leq 1$, with rate constant $\kk_{\yy \to \yy'}$, and 
        \item for any $\yy \to \yy' \in \mc R$ with $|\supp(\yy)| \geq 2$ and any $i \in \supp(\yy)$, include the reactions
            \eq{ 
                y_i \ee_i \to \yy' + \yy - y_i \ee_i ,
            }
        with rate constant $\kk_{\yy \to \yy'}  (\xx^*)^{\yy}/(x^*_i)^{y_i}$. 
    \end{enumerate}
    Let $\tilde{\mc V}$ be the set of source and target complexes of the reactions in $\tilde{\mc R}$. Then $\Gtilde = (\tilde{\mc V}, \tilde{\mc R})$ is the \df{modified network} of $\mc N$, and $\Gktilde$ is the \df{modified mass-action system} at $\xx^*$.  
\end{defn}

\begin{prop}
\label{prop:modifiedsystem}
    Let $\Gkde$ be a delay mass-action system, and $\tilde{\Jac}$ be the modified Jacobian matrix \eqref{eq:Jtilde} evaluated at some state $\xx^* > \vv 0$. Let $\Gktilde$ be the modified mass-action system at $\xx^*$. Then $\tilde{\Jac}$ is the Jacobian matrix of $\Gktilde$ evaluated at $\xx^*$. 
\end{prop}

The proof of this proposition follows by treating each reaction as in the sample calculation above, and noting that the Jacobian matrix is linear with respect to reactions.

\section{Non-injectivity and the DSR graph}
\label{sec:delay-SRintro}

The aim of this section is to introduce the \emph{directed species-reaction graph} (\emph{\SR})~\cite{CraciunBanaji2010}, and in \Cref{sec:delay-SR} we provide conditions on the \SR s of the original or modified network sufficient to conclude delay stability.

The \SR, and its closely related cousin the \emph{species-reaction graph} were used to study \emph{injectivity} of a reaction network, i.e., the Jacobian keeps the same sign for any positive rate constants and at any positive state~\cite{CraciunFeinberg2005, CraciunFeinberg2006, CraciunBanaji2009, CraciunBanaji2010}. In particular, injectivity can be used to rule out the capacity for \emph{multistationarity}, since an injective network cannot admit multiple positive steady states for any choice of rate constants. There are other variants of the \SR, including Petri nets from computer science and Volpert's graph for chemical reactions. We focus on a version that is a hybrid of what is defined in  \cite{CraciunBanaji2010} and \cite{CraciunFeinberg2006}.


Recall that a network with no one-step catalysis is one where $\supp(\yy)\cap\supp(\yy') = \emptyset$ for any reaction $\yy \to \yy'$. With such a network, the \SR\ takes on a simpler form. Here, we give a definition of the directed species-reaction graph that is sufficient for our purpose. We first illustrate how a \SR\ is drawn for a given network.

\begin{ex}
\label{ex:SR-intro}
Consider the reaction network
    \tikzinline{1}{
    \node (xy) at (-2.75,0) [right] {$\cf{X}+\cf{Y}$};
        \node (2z) at (0,0) [right] {$\cf{2Z}$};
        \node (z) at (2,0) [right] {$\cf{Z}$};
        \node (y) at (2,-1.5) [right] {$\cf{Y}$};
        \node (x) at (4,0) [right] {$\cf{X}$};
        \node (0) at (4,-1.5) [right] {$\cf{0}$};
        \draw [fwdrxn] (xy) -- (2z) node [above]{\ratecnst{\vphantom{$\kk_1$}}};
        \draw [fwdrxn] (z) -- (x);
        \draw [revrxn, revrxnNW] (y) -- (x);
        \draw [revrxn, revrxnSE] (x) -- (y);
        \draw [fwdrxn] (0) -- (y);
        \draw [revrxn, revrxnW] (0) -- (x);
        \draw [revrxn, revrxnE] (x) -- (0);
    }
with three species and five edges, two of which are reversible pairs. Two reactions ($\cf{0}\to\cf{X}$ and $\cf{0}\to\cf{Y}$) are \emph{inflows} and will not appear in \SR, likewise the \emph{outflow} reaction $\cf{X}\to\cf{0}$ will also not appear. The \SR\ of this reaction network, shown in \Cref{fig:SR-intro}, contains three species nodes and three reaction nodes. For the irreversible reaction $\cf{X+Y}\to\cf{2Z}$, the source species $\cf{X}$ and $\cf{Y}$ are connected to the reaction node by undirected edges, which are labelled with their \emph{stoichiometric coefficients} ($1$ and $1$ respectively). The product species $\cf{Z}$ receives an incoming edge from the reaction node, labelled with its stoichiometric coefficient $2$. For a reversible reaction like that of $\cf{X} \RR \cf{Y}$, the edges connecting the reaction node and the corresponding species nodes are undirected. In a \SR, any undirected edge should be understood as bidirectional.
\end{ex}

\begin{figure}[h!tbp] 
\centering 
    \begin{tikzpicture}
    \node[Snode] (x) at (0,1) {$\cf{X}$};
    \node[Snode] (y) at (0,-1) {$\cf{Y}$};
    \node[Snode] (z) at (5.5,0) {$\cf{Z}$};
    \node[Rnodelong] (r1) at (2.5,0) {$\cf{X}+\cf{Y} \to 2\cf{Z}$};
    \node[Rnodelong] (r2) at (2.5,2.5) {$\cf{Z} \to \cf{X}$};
    \node[Rnodelong] (r3) at (-2.25,0) {$\cf{X} \RR \cf{Y}$};
    \draw [SR, highlightorange] (x.east)--([srNW] r1.north west) node [midway, above] {\footnotesize $1$};
    \draw [SR, highlightorange] (y.east)--([srSW] r1.south west) node [midway, below] {\footnotesize $1$};
    \draw [DSR] (r1.east)--(z.west) node [midway, above] {\footnotesize $2$};
    \draw [SR] (y.west) to [bend left] node [midway, below] {\footnotesize $1$} ([xshift=0.5cm] r3.south);
    \draw [SR] ([xshift=0.5cm] r3.north)  to [bend left] node [midway, above] {\footnotesize $1$} (x.west);
    \draw [SR] (z.north) to [bend right] node[midway, above] {\footnotesize $1$} (r2.east) ;
    \draw [DSR] (r2.west) to [bend right] node [midway, above] {\footnotesize $1$} (x.north);

    \node[Snode] (x) at (0,1) {$\cf{X}$};
    \node[Snode] (y) at (0,-1) {$\cf{Y}$};
    \node[Snode] (z) at (5.5,0) {$\cf{Z}$};
    \node[Rnodelong] (r1) at (2.5,0) {$\cf{X}+\cf{Y} \to \cf{2Z}$};
    \node[Rnodelong] (r2) at (2.5,2.5) {$\cf{Z} \to \cf{X}$};
    \node[Rnodelong] (r3) at (-2.25,0) {$\cf{X} \RR \cf{Y}$};
\end{tikzpicture}
    \caption[Directed species-reaction (DSR) graph]{The \SR\ of \Cref{ex:SR-intro,ex:SR-cycles}. The edges (\orange{orange}) connecting the S-nodes $\cf{X}$ and $\cf{Y}$ to the R-node of the bispecies reaction is a \emph{c-pair} (\Cref{def:SR-cycles}).}
    \label{fig:SR-intro}
\end{figure}
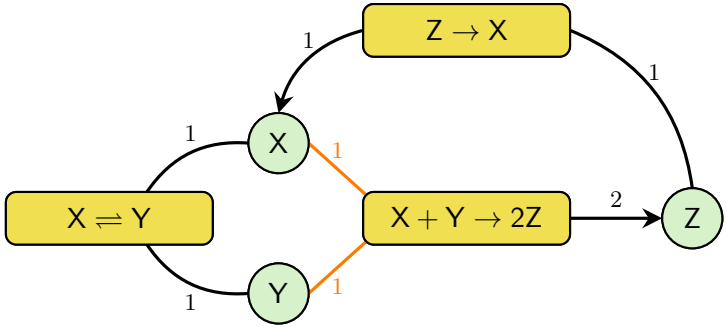

We now define the \df{directed species-reaction graph} (or \df{\SR}) of a reaction network with no one-step catalysis. The \SR\ is a bipartite graph with two kinds of nodes: \df{S-nodes} (one for each species) and \df{R-nodes} (one for each reaction that is neither a generalized inflow nor a generalized outflow). A reaction is a \df{generalized inflow} if it is of the form $\cf{0} \to a_i \cf{X}_i$ for some $a_i > 0$. It is a \df{generalized outflow} if it is of the form $a_i \cf{X}_i \to \cf{0}$ for some $a_i > 0$. To simplify the language, we refer to a S-node as if it is the species, and a R-node as if it is the relevant reaction; for example, an R-node $\cf{R}$ is irreversible if the corresponding reaction in the network is irreversible. An edge in the \SR\ will be denoted by the ordered pair in $V_S \times V_R$; whether it is directed or not will be explicitly stated. An edge connects a S-node $\cf{X}_i$ with a R-node $\cf{R}$ if and only if the species $\cf{X}_i$ participates as a reactant or a product in the corresponding reaction (pair) $\cf{R}$. The edge is \emph{directed} towards $\cf{X}_i$ if the reaction is \emph{irreversible} and $\cf{X}_i$ is a product in $\cf{R}$. Finally, each edge is labelled with the \df{stoichiometric coefficient} of the species in the reaction $\cf{R}$. 

In this work, a \SR\ is always associated to an underlying reaction network. However, it can also be viewed as a graph in its own right. Notably, a \SR\ $\DSR = (V_S, V_R, E, \sigma)$ is a bipartite, partially directed graph $(V_S, V_R, E)$ with a map $\sigma\colon E \to \rrpp$. Vertices in $V_S$ are S-nodes, while those in $V_R$ are R-nodes. Each edge in $E$, either directed or undirected, is assigned a stoichiometric coefficient by the map $\sigma$.

\begin{rmk}
    \SR\ is defined for more general reaction networks. If the network has a one-step catalysis, its \SR\ is a multigraph~\cite{CraciunFeinberg2006, CraciunBanaji2010}. For example, the \SR\ of the reaction $\cf{X}+\cf{Y} \FR \cf{2X}$ has two edges between the S-node $\cf{X}$ and the R-node, one of which is not directed and assigned a stoichiometric coefficient $1$, while the other is directed and assigned a coefficient $2$. In this work, we avoid any one-step catalysis. 
\end{rmk}

The aim is to deduce information about the determinant of the Jacobian matrix from the \SR~\cite{CraciunFeinberg2006, CraciunBanaji2010, CraciunBanaji2009}. Cycles in the \SR\ are of interest to us. A subset of edges defines a subgraph of the \SR. A \emph{path} is an open simple walk compatible with the orientation of all edges it contains (undirected edges are thought of as allowing either orientations), and a \emph{cycle} is a closed simple walk compatible with the orientation of all of its eges. If $C_1$ and $C_2$ are two cycles, we say the intersection $C_1\cap C_2$ is non-empty if there is at least one edge in $C_1 \cap C_2$, and the orientation of every edge in $C_1\cap C_2$ is consistent with that of $C_1$ and also that of $C_2$.

\begin{defn}
\label{def:SR-cycles}
    Let $\DSR = (V_S, V_R, E, \sigma)$ be the \SR\ of a reaction network with no one-step catalysis. 
    \begin{enumerate}[label={\textup{(\alph*)}}]
    \item 
        A \df{c-pair} (\df{complex-pair}) is a pair of edges adjacent to a R-node such that the two adjacent species are reactants in the reaction corresponding to that R-node.  
    \item 
        A cycle is an \df{e-cycle} (\df{even-cycle}) if it contains an even number of c-pairs. 
    \item 
        A cycle is an \df{o-cycle} (\df{odd-cycle}) if it contains an odd number of c-pairs. 
    \item 
        Alternatingly multiply and divide the stoichiometric coefficients along a cycle. If the result is equal to $1$, then the cycle is an \df{s-cycle}.
    \item  
        Two cycles $C_1$ and $C_2$ have \df{S-to-R intersection} if $C_1\cap C_2$ is non-empty and all  connected components of $C_1\cap C_2$  are paths of odd length, i.e., each starting at an S-node and terminating at an R-node, or  starting at an R-node and terminating at an S-node. We say $\DSR$ \df{has S-to-R intersection} if there exist two cycles in $\DSR$ with S-to-R intersection.
    \end{enumerate}
\end{defn}

\begin{ex}
\label{ex:SR-cycles}
We revisit the \SR\ in \Cref{fig:SR-intro}. The coloured (orange) edges connecting the S-nodes $\cf{X}$ and $\cf{Y}$, to the R-node of the irreversible bispecies reaction, form a c-pair. There are three cycles in the \SR: the left-most cycle $C_1$ contains only the S-nodes $\cf{X}$ and $\cf{Y}$; the upper-right cycle $C_2$ contains only the S-nodes $\cf{X}$ and $\cf{Z}$; and  running along the outer edges of the graph, $C_3$ contains all three S-nodes.
The left-most cycle $C_1$ passing through only $\cf{X}$ and $\cf{Y}$ is a s-cycle and an o-cycle, since all stoichiometric coefficients are $1$ and $C_1$ contains the c-pair. The cycle $C_2$ passing through only $\cf{X}$ and $\cf{Z}$ is \emph{not} a s-cycle, but it is an e-cycle. The cycle $C_3$ passing through all S-nodes is an e-cycle but not a s-cycle. Finally, $C_1$ and $C_2$ have a S-to-R intersection, since $C_1\cap C_2$ is half of the c-pair inheriting the directions of the two cycles. Similarly, $C_1$ and $C_3$ have a S-to-R intersection, and $C_2$ and $C_3$ also have S-to-R intersection.
\end{ex}

Cycles in the \SR\ are intimately connected to the principal minors of the Jacobian matrix~\cite{CraciunFeinberg2006, CraciunBanaji2009, CraciunBanaji2010, BanajiRutherford2011, Banaji2012}.

\begin{thm}[\cite{CraciunFeinberg2006,CraciunBanaji2010}]
\label{thm:SR-cf2006}
    Let $\mc N$ be a network with no one-step catalysis, and there is a generalized outflow for every species. Suppose its \SR\ satisfies the following conditions:
    \begin{enumerate}[label={\textup{(\alph*)}}]
        \item all cycles are o-cycles or s-cycles, and 
        \item no two e-cycles have S-to-R intersection. 
    \end{enumerate}
    Then the negative Jacobian matrix $-\Jac$ of the mass-action system $\Gk$ is a $P_0$-matrix for any $\vv\kk > \vv 0$ and $\xx > \vv 0$. Moreover, $\det(-\Jac) > 0$. 
\end{thm}

The result in \cite{CraciunBanaji2010} is slighly stronger than the version above: under the same hypothesis, $-\Jac$ is a $P$-matrix, i.e. it has positive principal minors. We also note that, moreover, all non-zero minors of $-\Jac$ are polynomials in $\vv\kk$ and $\vv\xx^*$~\cite{CraciunFeinberg2005} with positive coefficients. 

In \Cref{cor:delay-alg}, we presented an algebraic condition for delay stability based on the modified Jacobian matrix $\tilde{\Jac}$. We then related $\tilde{\Jac}$ to the Jacobian matrix of a different mass-action system in \Cref{prop:modifiedsystem}. Therefore by \Cref{thm:SR-cf2006}, delay stability follows if the modified mass-action system has a \SR\ that satisfies the above graph-theoretic conditions. 
Furthermore, our construction of the modified network always results in reactions involving at most one reactant species, so there is \emph{no} c-pair in the \SR\ of the modified network, and all cycles are e-cycles. This simplifies \Cref{thm:SR-cf2006}.

\begin{cor}
\label{cor:SR-cf2006}
    Let $\mc N$ be a network with no one-step catalysis, and there is a generalized outflow for every species. For any choice of rate constants $\vv\kk > \vv 0$, let $\Gk$ denote the mass-action system, and let $\Gktilde$ be its modified mass-action system at any positive state $\xx > \vv 0$, with Jacobian matrix $\tilde{\Jac}$. Suppose the \SR\ $\tilde{\DSR}$ of $\tilde{\mc N}$ satisfies the following conditions:
    \begin{enumerate}[label={\textup{(\alph*)}}]
        \item all cycles are s-cycles, and 
        \item there is no S-to-R intersection.
    \end{enumerate}
    Then $-\tilde{\Jac}$ is a $P_0$-matrix, and $\det(-\tilde{\Jac}) > 0$ for any $\vv\kk > \vv 0$ and $\xx > \vv 0$.
\end{cor}

From the outflows for every species, it also follows that $\tilde{\Jac}_{ii} < 0$. The \SR\ of $\tilde{\mc N}$ does not provide information about the Jacobian $\det\Jac$ of the original network; however, we will see in \Cref{thm:delay-SR} that under some mild assumptions, we know that $\det\Jac \neq 0$. In particular, \Cref{thm:delay-SR} implies delay stability for $\mc N$.

\section{DSR-graph condition for delay stability}
\label{sec:delay-SR}

In this section, we relate the \SR s of a reaction network and that of the modified network. In general, the modified network has many more reactions than the original, and its \SR\ contains more nodes and cycles. Moreover, the modified network is an artifact of \Cref{prop:modifiedsystem}; it may not have any biological relevance. The aim is to deduce delay stability based on the structure of the \SR\ of the original network instead of that of the modified network.

For the main result of this section, we assume the (original) reaction network has the followig four properties:
\begin{enumerate}[label={\bf{(N\arabic*)}}, leftmargin=1.25cm]
\item 
    Each species has a generalized outflow, i.e., $a_i\cf{X}_i \to 0$ for some $a_i > 0$ is a reaction for all $i$. 
\item 
    The network has no one-step catalysis, i.e., $\supp(\yy)\cap \supp(\yy') = \emptyset$ for any reaction $\yy \FR \yy'$.
\item 
    Every reaction has at most two different reactant species, i.e., $|\supp(\yy)| \leq 2$ for any reaction $\yy \to \yy'$.
\item 
    Any bispecies reaction is irreversible. 
\end{enumerate}

Many reasonable biochemical systems satisfy \textbf{(N1)}--\textbf{(N3)}. Condition \textbf{(N1)} typically reflects the natural degradation of molecules. Our results do not apply to autocatalytic systems. According to Condition \textbf{(N2)}, no species can participate both as a reactant and as a product of the same reaction. Condition \textbf{(N3)} is similar to, but more relaxed than, the common assumption that a reaction requiring three participating molecules is a rare event and can be safely neglected from the model. Finally, although condition \textbf{(N4)} seems stringent, it is not unreasonable when considering the types of systems for which delays are appropriate. For example, if monomers are abundantly available, polymerization reactions, going through a chain of intermediates, may be approximated as a single irreversible reaction from initiation to termination with delay. Moreover, our model, which uses discrete delay instead of \emph{distributed delay}, may not be appropriate in the near-thermodynamic equilibrium regime where reactions can be considered reversible~\cite{Epstein1990, Roussel1996, RousselRoussel2001}.

\begin{rmk}
    Condition \textbf{(N1)} can be replaced by the following {\em less restrictive} condition:
\begin{enumerate}[leftmargin=1.25cm]
\item[\textbf{(N1')}]
    There exists a choice of $n$ reactions $\yy_1 \to \yy_1'$, $\ldots$, $\yy_n \to \yy_n'$ such that 
    \eq{ 
        \det(\yy_1, \cdots, \yy_n) \det(\yy_1 - \yy_1', \cdots , \yy_n - \yy_n') > 0.
    }
\end{enumerate}
\end{rmk}

\bigskip

The remainder of this section is structured as follows. First we state the main results; then apply them to several networks. Next we illustrate by way of examples the difference between the \SR s of the original and modified networks. Finally we state and prove a series of lemmas, leading to a proof of the main theorem.

\begin{defn}
\label{def:bispeciesedge}
    Let $\cf{R}$ be an R-node of an irreversible reaction involving two reactant species. Let $\cf{X}$ be an S-node corresponding to one of the \emph{product} species of the reaction. The directed edge $(\cf{X}, \cf{R})$ in $\DSR$ is called a \df{bispecies production edge}.
\end{defn}

\begin{figure}[h!tbp]
\centering
    \begin{tikzpicture}
    \node at (0,3.3) {};
    \node[Snode] (x) at (0.25,1) {$\cf{X}$};
    \node[Snode] (y) at (0.25,-1) {$\cf{Y}$};
    \node[Snode] (z) at (5.25,0) {$\cf{Z}$};
    \node[Rnodelong] (r1) at (2.5,0) {$\cf{X}+\cf{Y} \to \cf{Z}$};
    \node[Rnodelong] (r2) at (2.5,2.75) {$\cf{Z} \to \cf{X}$};
    \draw [SR] (x.south east)--([srNW] r1.north west) node [midway, above] {\footnotesize };
    \draw [SR] (y.north east)--([srSW] r1.south west) node [midway, below] {\footnotesize };
    \draw [DSR] (r1.east)--(z.west) node [midway, above] {\footnotesize };
    \draw [SR] (z.north) to [bend right] node[midway, above] {\footnotesize } (r2.east) ;
    \draw [DSR] (r2.west) to [bend right] node [midway, above] {\footnotesize } (x.north);
    
   \node[Snode] (x) at (0.25,1) {$\cf{X}$};
    \node[Snode] (y) at (0.25,-1) {$\cf{Y}$};
    \node[Snode] (z) at (5.25,0) {$\cf{Z}$};
    \node[Rnodelong] (r1) at (2.5,0) {$\cf{X}+\cf{Y} \to \cf{Z}$};
    \node[Rnodelong] (r2) at (2.5,2.75) {$\cf{Z} \to \cf{X}$};
\end{tikzpicture}
    \caption{A \SR\ containing a cycle with a bispecies production edge. The bispecies production edge is the directed edge leading into $\cf{Z}$.}
    \label{fig:SRIntxn-bisp-prodt}
\end{figure}
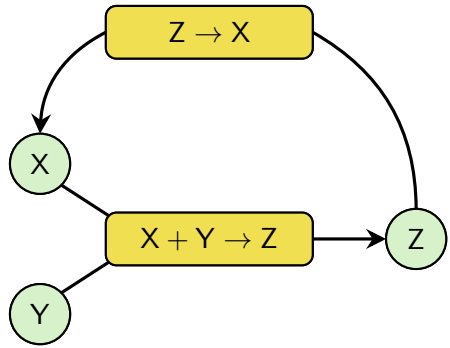

To conclude delay stability, we would like to avoid cycles with bispecies production edges. Such a cycle can be interpreted as a feedback loop. For example, the reaction network 
    \tikzinline{1}{
        \node (xy) at (0,-1.5) [left] {$\cf{X}+\cf{Y}$};
        \node (z) at (1.5, -1.5) [right] {$\cf{Z}$};
        \node (x) at (3.5,-1.5) [right] {$\cf{X}$};
        \draw [fwdrxn] (xy) -- (z);
        \draw [fwdrxn] (z) -- (x);
    }
whose \SR, shown in \Cref{fig:SRIntxn-bisp-prodt}, has a cycle with a bispecies production edge. This cycle points out that the product $\cf{Z}$ of a bispecies reaction is eventually used to feed the production of reactant $\cf{X}$. 

The main results of this section are the following.  The proof of \Cref{thm:SRmodified} can be found after \Cref{ex:CST}.

\begin{thm}
\label{thm:SRmodified}
    Let $\mc N$ be a reaction network satisfying conditions \textbf{(N2)}--\textbf{(N4)}. For any vector of rate constants $\vv\kk > \vv 0$, let $\Gk$ denote the mass-action system, and let $\Gktilde$ be its modified mass-action system at any positive state. In the \SR\ $\tilde{\DSR}$ of the network $\Gtilde$ we have
    \begin{enumerate}[label={\textup{(\alph*)}}]
        \item[\textup{(i)}] 
            all cycles are s-cycles, and 
        \item[\textup{(ii)}] 
            there is no S-to-R intersection, 
    \end{enumerate}
    if and only if in the \SR\ $\DSR$ of the network $\mc N$ we have 
    \begin{enumerate}[label={\textup{(\alph*)}}]
        \item 
            no cycle contains a bispecies production edge; 
        \item 
            all cycles are s-cycles, and 
        \item 
            there is no S-to-R intersection. 
    \end{enumerate}
\end{thm}

\begin{thm}
\label{thm:delay-SR}%
    Let $\mc N$ be a reaction network satisfying conditions \textbf{(N1)}--\textbf{(N4)}. Suppose $\DSR$, the \SR\ of $\mc N$, satisfies the following:
    \begin{enumerate}[label={\textup{(\alph*)}}]
        \item 
            no cycle contains a bispecies production edge; 
        \item 
            all cycles are s-cycles, and 
        \item 
            there is no S-to-R intersection. 
    \end{enumerate}
    Then $\mc N$ is delay stable, i.e., for any rate constants $\vv\kk > \vv 0$ and any delay parameters $\vv\tau \geq \vv 0$, any equilibrium of the delay mass-action system $\Gkde$ is asymptotically stable. 
\end{thm}
\begin{proof}   
    Let $\mc N$ be a reaction network satisfying \textbf{(N1)}--\textbf{(N4)}, and suppose its \SR\ $\DSR$ satisfies (a)--(c) of \Cref{thm:delay-SR}. Let $\Gtilde$ be the modified network as in \Cref{def:modifiedMAS}, and $\tilde{\DSR}$ be its \SR. By \Cref{thm:SRmodified}, all cycles in $\tilde{\DSR}$ are s-cycles, and $\tilde{\DSR}$ has no S-to-R intersection. By \Cref{cor:SR-cf2006}, its Jacobian matrix $-\tilde{\Jac}$ is a $P_0$-matrix, independent of the choice of rate constant $\vv\kk > \vv 0$ and the positive state. \Cref{prop:modifiedsystem} says that $\tilde{\Jac}$ is the modified Jacobian matrix of the delay mass-action system $\Gkde$. 
    
    It remains to be shown --- independent of $\vv \kk > \vv 0$, $\vv x > \vv 0$ --- that $[\tilde{\Jac}]_{pp} < 0$ for all $1 \leq p \leq n$ and $\det(\Jac) \neq 0$, where $\Jac$ is the Jacobian matrix of the mass-action system $\Gk$. Condition \textbf{(N1)} --- that every species $\cf{X}_i$ has a generalized outflow reaction $a_i\cf{X}_o \to \cf{0}$ (with rate constant $\kk_{i-}$) --- guarantees that $[\tilde{\Jac}]_{pp} = [\Jac]_{pp} < 0$ since 
    \eq{ 
       -[\tilde{\Jac}]_{pp} = -[\Jac]_{pp}
       = -\sum_{r \in \mc R} \kk_{r} \frac{\partial \xx^{\yy_r}}{\partial x_p} (0 - y_{rp})
       \geq \kk_{i-} \frac{\partial x_p^{a_i}}{\partial x_p} a_i > 0.
    }
     Finally, the conditions on the \SR\ of $\mc N$ also ensures that $-\Jac$ itself is a $P_0$-matrix and that $\det(-\Jac) > 0$ by \Cref{thm:SR-cf2006}. Therefore, delay stability of $\mc N$ follows from \Cref{cor:delay-alg}. 
\end{proof}

\begin{rmk}\label{rmk:N1'}
    In a previous remark we claimed that the condition \textbf{(N1)} can be replaced by the less restrictive condition \textbf{(N1')}, and the statement of \Cref{thm:delay-SR} still holds. Let us see why that is the case. Notice above that we have used \textbf{(N1)} in two places in the proof of \Cref{thm:delay-SR}: first to guarantee that  $[\tilde{\Jac}]_{pp} < 0$, and then (implicitly) to guarantee that $\det(-\Jac) > 0$. Indeed, \textbf{(N1')} together with \textbf{(N2)} imply that  $[\tilde{\Jac}]_{pp} < 0$, and also \textbf{(N1')} together with conditions (b) and (c) imply $\det(-\Jac) > 0$ (according to \cite{CraciunFeinberg2005, CraciunFeinberg2006}). 
\end{rmk}

\begin{rmk}
    Conditions (b) and (c) of \Cref{thm:delay-SR} imply that there is at most one positive equilibrium, since these conditions imply that reaction network $\mc N$ is \emph{injective}~\cite{CraciunFeinberg2005}, which in turn rules out multiple positive equilibria. 
\end{rmk}

\begin{rmk}
    \Cref{thm:delay-SR} is applicable to mass-action systems \emph{without delays}, by taking $\vv \tau = \vv 0$. In particular, if the conditions in \Cref{thm:delay-SR} are met, then for any choice of rate constants $\vv\kk > \vv 0$, any positive equilibrium of the mass-action system $\mc N_{\vv\kk}$ is asymptotically stable.
\end{rmk}

Since the proof of \Cref{thm:SRmodified} is very technical and does not shed light on the underlying structure on the \SR, it suffices for now to say that (1) the s-cycles of $\DSR$ and $\tilde{\DSR}$ are related (\Cref{lem:SR-hom-label,prop:SR-hom-s-cycle}); (2) there is a S-to-R intersection in $\tilde{\DSR}$ if and only if either there is one in $\DSR$ or there is a cycle containing a bispecies production edge (\Cref{prop:SR-hom-SRintxn}). So instead, we explore several examples applying \Cref{thm:delay-SR} before proving \Cref{thm:SRmodified}.

Several physical and biochemical processes are believed to follow a \emph{nucleation-propagation mechanism}, from crystallization in solution, polymerization reactions including micelles formation~\cite{NuclProp-micelle, NuclProp-polymer}, and DNA double-helix formation~\cite{NuclProp-DNA, NuclProp-DNA2, NuclProp-DNA3}. Under this mechanism, nucleation, whereby the process is initiated, is the rate determining step, followed by the faster propagation step that takes the process to completion.

\begin{ex}
\label{ex:DNA-SR}
We continue with \Cref{ex:DNA,ex:DNA-char}, the formation of double-stranded DNA.
Of course, one can check the criteria of \Cref{cor:delay-alg}, which is not difficult since the modified Jacobian matrix (which also happens to be the Jacobian matrix)
    \eq{ 
        \tilde{\mm J} = \begin{pmatrix}
            -\kk_5 - 4 \kk_1 \cf{S}  & 2\kk_2  \\
        2\kk_1 \cf{S}  & -\kk_2 -\kk_3 
        \end{pmatrix}
    }
is only a $2\times 2$ matrix. However, we look instead at its \SR, shown in \Cref{fig:ex:DNA}. The \SR\ has no cycle, so conditions (a)--(c) of \Cref{thm:delay-SR} are trivially satisfied. Therefore, the network corresponding to duplex formation is delay stable, i.e., asymptotically stable for any choice of rate constants and delay parameters.  
\end{ex}

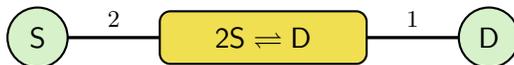
\begin{figure}[h!tbp]
\centering
    \begin{tikzpicture}
    \node at (0,-1) {};
    \node at (0,1) {};
    
    \node[Snode] (x) at (-3,0) {$\cf{S}$};
    \node[Snode] (y) at (3,0) {$\cf{D}$};
    \node[Rnodelong] (r) at (0,0) {$\cf{2S}\RR \cf{D}$};
    \draw [SR] (x.east)--(r.west) node [midway, above] {\footnotesize $2$};
    \draw [SR] (y.west)--(r.east) node [midway, above] {\footnotesize $1$};

\end{tikzpicture}
    \caption[\SR\ for the formation of double-stranded DNA via nucleation-propagation mechanism]{The \SR\ of \Cref{ex:DNA} has no cycle. Thus, the toy model for the formation of double-stranded helix via the nucleation-propagation mechanism is delay stable.
    } 
    \label{fig:ex:DNA}
\end{figure}

\begin{rmk}
\Cref{thm:delay-SR} applies whenever the \SR\ of the reaction network contains \emph{no cycle} whatsoever, and the network satisfies \textbf{(N1)}--\textbf{(N4)}. If the \SR\ of the network is a single cycle, then the conditions of \Cref{thm:delay-SR} amount to simply asking that the cycle is an s-cycle and that it does not contain a bispecies production edge. An example of such a class of networks is discussed next.
\end{rmk}

\begin{ex} 
\label{ex:CST}
A  CST network (cyclic sequestration-transmutation network) on $n$ species $\cf{X}_1, \ldots, \cf{X}_n$ has $n$ reactions $\cf{R}_1,\ldots, \cf{R}_n$, where  each $\cf{R}_i$ is either 
    \begin{align}\label{eq:CSTnet}
    &a_i \cf{X}_i + b_{i+1} \cf{X}_{i+1} \to \cf{0} \text{ (sequestration) \quad or}\\
    &a_i \cf{X}_i \to b_{i+1} \cf{X}_{i+1} \text{ (transmutation)}\nonumber
    \end{align}
with $\cf{X}_{n+1}=\cf{X}_1$.  

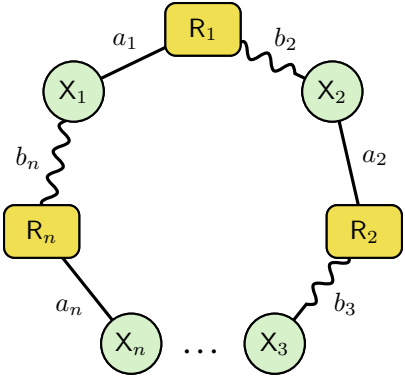
\begin{figure}
\begin{center}
\begin{tikzpicture}[scale=1.1,font=\small,reaction/.style={rectangle, minimum size=5mm, inner sep=2pt,  draw=red!50!black!50, top color=red!20!black!10, bottom color=red!50!black!20,font=\small}, species/.style={rectangle,draw=white!5,fill=white!9,thick,inner sep=0pt,minimum size=3mm, font=\small}, arc/.style={shorten <=2pt,shorten >=2pt, >=stealth',semithick}]
 \node[Snode, outer sep=-2pt] (X1) at (141.4:2) {$X_1$};
 \node[Snode, outer sep=-2pt] (X2) at (38.6:2) {$X_2$};
 \node[Snode, outer sep=-2pt] (X3) at (-64.3:2) {$X_3$};
 \node[Snode, outer sep=-2pt] (Xn) at (244.3:2) {$X_n$};
 \node[Rnode, outer sep=-2pt] (1) at (90:2) {$R_1$}
 edge[arc, -, very thick] node[xshift=-3pt, yshift=8pt]{$a_1$} (X1);
 \draw[decorate, decoration={coil,aspect=0}, very thick] (1)--(X2) node[midway, xshift=5pt, yshift=9pt] {$b_2$};
\node[Rnode, outer sep=-2pt] (n) at (192.85:2) {$R_n$}
edge[arc, -, very thick] node[midway, below left,xshift=3pt]{$a_n$} (Xn);
\draw[decorate, decoration={coil,aspect=0}, very thick] (n)--(X1) node[midway, xshift=-11pt, yshift=2pt] {$b_n$};
\node[Rnode, outer sep=-2pt] (2) at (-12.85:2) {$R_2$}
edge[arc, -, very thick] node[xshift=10pt, yshift=2pt]{$a_2$} (X2);
\draw[decorate, decoration={coil,aspect=0}, very thick] (2)--(X3)
node[midway, xshift=10pt, yshift=-6pt] {$b_3$};
\node at (-90:1.9) {$\boldsymbol{\cdots}$};
 \node[Snode] (X1) at (141.4:2) {$\cf{X}_1$};
 \node[Snode] (X2) at (38.6:2) {$\cf{X}_2$};
 \node[Snode] (X3) at (-64.3:2) {$\cf{X}_3$};
 \node[Snode] (Xn) at (244.3:2) {$\cf{X}_n$};
 \node[Rnode] (1) at (90:2) {$\cf{R}_1$};
 \node[Rnode] (2) at (-12.85:2) {$\cf{R}_2$};
 \node[Rnode] (n) at (192.85:2) {$\cf{R}_n$};
\end{tikzpicture}
\caption{DSR graph of the general CST network.}
\label{fig:cstdsr}
\end{center}
\end{figure}

The DSR graph of a CST network is characterized by the single-cycle structure given in \Cref{fig:cstdsr}. Here the wavy edge adjacent to reaction $\cf{R}_i$ is either a directed arrow from $\cf{R}_i$ to $\cf{X}_{i+1}$ if $\cf{R}_i$ is a transmutation, or an undirected edge which is part of a c-pair if $\cf{R}_i$ is a sequestration. These types of cycles are common in DSR graphs, which makes CST networks useful as ``motif'' whose dynamical properties may be inherited by the large network.  For example, \cite{CST} involves inheriting multistationarity.  

A \emph{fully open CST} is a CST network together with inflow and outflow reactions for all species. Fully open CST networks satisfy conditions \textbf{(N1)}--\textbf{(N4)}, and therefore \Cref{thm:delay-SR} implies that if $a_1\cdots a_n=b_1\cdots b_n$ (i.e., the cycle is an s-cyle), then the fully open CST network is delay stable.

We note that by replacing condition \textbf{(N1)}  (in an equivalent way, as explained in \Cref{rmk:N1'}) with \textbf{(N1')}, a more powerful conclusion can be drawn: it is enough to require that at least one species (instead of all) is in the outflow to conclude delay stability. Indeed, if we assume without loss of generality that $\cf{X}_n$ has an outflow reaction, then by replacing $\cf{R}_n$ in (\ref{eq:CSTnet}) by $\cf{X}_n \to \cf{0}$  we obtain the reactions $\yy_1\to \yy'_1$, $\ldots$, $\yy_n\to \yy'_n$ satisfying
    \eq{
        \det\left( \vphantom{\sum} \yy_1, \cdots, \yy_n\right)
        \det\left( \vphantom{\sum} \yy_1-\yy'_1,\cdots, \yy'_n-\yy_n \right)
        =(a_1a_2\cdots a_{n-1})^2 > 0.
    }
Therefore condition \textbf{(N1')} is satisfied, and the network is delay stable.
\end{ex}

We now return to a discussion leading up to a proof of \Cref{thm:SRmodified}. One can imagine if the reaction network $\mc N$ has only one bispecies reaction, then $\mc N$ differs from its modified network $\Gtilde$ only by the relevant reaction (two reactions from the view of $\Gtilde$), and their respective \SR s differ only near the relevant R-node(s). Indeed, the Jacobian matrix $\Jac$ and the modified Jacobian matrix $\tilde{\Jac}$ are different only when there is an bispecies reaction.

\begin{figure}[h!tbp]
\centering
\begin{subfigure}[b]{0.37\textwidth}
    \centering 
    \begin{tikzpicture}
    \node[Snode] (x) at (0,1.25) {$\cf{X}$};
    \node[Snode] (y) at (0,-1.25) {$\cf{Y}$};
    \node[Snode] (z) at (5,0) {$\cf{Z}$};
    \node[Rnodelong] (r1) at (2.35,0) {$\cf{2X}+\cf{Y} \to \cf{Z}$};
    \draw [SR, highlightorange] (x.south east)--([srNW] r1.north west) node [midway, above right] {\footnotesize $2$};
    \draw [SR, highlightblue] (y.north east)--([srSW] r1.south west) node [midway, below right] {\footnotesize $1$};
    \draw [DSR] (r1.east)--(z.west) node [midway, above] {\footnotesize $1$};

    \node[Snode] (x) at (0,1.25) {$\cf{X}$};
    \node[Snode] (y) at (0,-1.25) {$\cf{Y}$};
    \node[Snode] (z) at (5,0) {$\cf{Z}$};
    \node[Rnodelong] (r1) at (2.35,0) {$\cf{2X}+\cf{Y} \to \cf{Z}$};
\end{tikzpicture}
    \caption{}
    \label{fig:SRmodified-local-a}
\end{subfigure}\hspace{0.1\textwidth}
\begin{subfigure}[b]{0.37\textwidth}
    \centering 
    \begin{tikzpicture}
    \node[Snode] (x) at (0,1.25) {$\cf{X}$};
    \node[Snode] (y) at (0,-1.25) {$\cf{Y}$};
    \node[Snode] (z) at (5,0) {$\cf{Z}$};
    \node[Rnodelong] (r1) at (2.35,1.25) {$\cf{2X} \to \cf{Z}+\cf{Y}$};
    \node[Rnodelong] (r2) at (2.35,-1.25) {$\cf{Y} \to \cf{Z}+\cf{2X}$};
    \draw [SR, highlightorange] (x.east)--(r1.west) node [midway, above] {\footnotesize $2$};
    \draw [SR, highlightblue] (y.east)--(r2.west) node [midway, below] {\footnotesize $1$};
    
    \draw [DSR, highlightblue] ([srSW] r1.south west)--(y.north east) node [near start, below right] {\footnotesize \!\!$1$};
    \draw [DSR, highlightorange] ([srNW] r2.west)--(x.south east) node [near start, above right] {\footnotesize \!\!$2$};

    \draw [DSR] (r1.east)--(z.north west) node [midway, above right] {\footnotesize $1$};
    \draw [DSR] (r2.east)--(z.south west) node [midway, below right] {\footnotesize $1$};

    \node[Snode] (x) at (0,1.25) {$\cf{X}$};
    \node[Snode] (y) at (0,-1.25) {$\cf{Y}$};
    \node[Snode] (z) at (5,0) {$\cf{Z}$};
    \node[Rnodelong] (r1) at (2.35,1.25) {$\cf{2X} \to \cf{Z}+\cf{Y}$};
    \node[Rnodelong] (r2) at (2.35,-1.25) {$\cf{Y} \to \cf{Z}+\cf{2X}$};
\end{tikzpicture}
    \caption{}
    \label{fig:SRmodified-local-b}
\end{subfigure}
    \caption{The \SR s from \Cref{ex:SRmodified-local}. (a) The original network $\{\cf{2X}+\cf{Y}\to\cf{Z}\}$ and (b) its modified network.}
    \label{fig:SRmodified-local}
\end{figure}

\begin{ex}
\label{ex:SRmodified-local}
Consider the reaction network consisting of a single reaction
    \tikzinline{1}{
        \node (xy) at (0,0) [left] {$\cf{2X}+\cf{Y}$};
        \node (z) at (1.5,0) [right] {$\cf{Z}$};
        \draw [fwdrxn] (xy) -- (z) node [above, midway] {\ratecnst{\vphantom{$\kk$}}};
    }
whose \SR\ $\DSR$ is shown in \Cref{fig:SRmodified-local-a}. Independent of the choice of rate constants and positive state, the modified network consisting of two reactions 
    \tikzinline{1}{
        \node (xy) at (0,0) [left] {$\cf{2X}$};
        \node (z) at (1.5,0) [right] {$\cf{Z}+\cf{Y}$};
        \draw [fwdrxn] (xy) -- (z) node [above, midway] {\ratecnst{\vphantom{$\kk$}}};
        \node (xy) at (0,-1) [left] {$\cf{Y}$};
        \node (z) at (1.5,-1) [right] {$\cf{Z}+\cf{2X}$};
        \draw [fwdrxn] (xy) -- (z);
    }
has the \SR\ $\tilde{\DSR}$ in \Cref{fig:SRmodified-local-b}. The edges are coloured in the figure to show correspondence with those in \Cref{fig:SRmodified-local-a}. Each R-node in $\tilde{\DSR}$ has degree $3$, and its adjacent edges have the same set of stoichiometric coefficients, which also happens to be the stoichiometric coefficient in \Cref{fig:SRmodified-local-a}. A s-cycle is present in $\tilde{\DSR}$ even though $\DSR$ has no cycle. 
    
On the \SR s in \Cref{fig:SRmodified-local}, we can define a graph homomorphism $\Phi \colon \tilde{\DSR} \to \DSR$. It acts as the identity on the S-nodes, and maps the two R-nodes in $\tilde{\DSR}$ to the R-node in $\DSR$. So $\Phi$ maps any edge of $\tilde{\DSR}$ correspondingly (edges colour-coded as such in \Cref{fig:SRmodified-local}). Note that the stoichiometric coefficient of any edge is preserved by $\Phi$. Finally, note that a directed edge in $\tilde{\DSR}$ may become undirected in $\DSR$ under the graph homomorphism $\Phi$. 
\end{ex}

Regarding the \SR s of a reaction network and its modified network, it is not difficult to imagine all the actions happening around R-nodes in $\DSR$ involving two reactant species. We call such a node a \df{bispecies R-nodes}.%

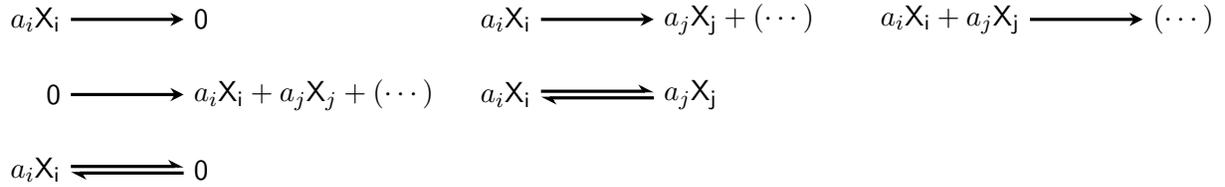
\begin{figure}[h!tbp] 
\centering 
    \begin{tikzpicture}
    \node at (0,0.75) {};
    \node at (0,-2.5) {};
\begin{scope}[transform canvas={xshift=-1.cm}]
\begin{scope}[transform canvas={xshift=-6.25cm}]
    \node (x) at (0,0) [left] {$a_i{\cf{X_i}}$};
    \node (y) at (1.5,0) [right] {$\cf{0}$};
    \draw [fwdrxn] (x) -- (y) ;
    \node (x) at (1.5,-1) [right] {$a_i{\cf{X_i}} + a_j\cf{X}_j + (\cdots)$};
    \node (y) at (0,-1) [left] {$\cf{0}$};
    \draw [fwdrxn] (y) -- (x) ;
    \node (x) at (0,-2) [left] {$a_i{\cf{X_i}}$};
    \node (y) at (1.5,-2) [right] {$\cf{0}$};
    \draw [revrxn, revrxnN] (x) -- (y) ;
    \draw [revrxn, revrxnS] (y) -- (x) ;
\end{scope}
\begin{scope}
    \node (x) at (0,0) [left] {$a_i{\cf{X_i}}$};
    \node (y) at (1.5,0) [right] {$a_j{\cf{X_j}} + \cf{(\cdots)}$};
    \draw [fwdrxn] (x) -- (y) ;
    \node (x) at (0,-1) [left] {$a_i{\cf{X_i}}$};
    \node (y) at (1.5,-1) [right] {$a_j{\cf{X_j}}$};
    \draw [revrxn, revrxnN] (x) -- (y) ;
    \draw [revrxn, revrxnS] (y) -- (x) ;
\end{scope}
\begin{scope}[transform canvas={xshift=6.5cm}]
    \node (x) at (0,0) [left] {$a_i{\cf{X_i}}+a_j{\cf{X_j}}$};
    \node (y) at (1.5,0) [right] {$\cf{(\cdots)}$};
    \draw [fwdrxn] (x) -- (y) ;
\end{scope}
\end{scope}
\end{tikzpicture}
    \caption[Admissible reactions for \Cref{thm:SRmodified,thm:delay-SR}]{A reaction network satisfying conditions \textbf{(N2)}--\textbf{(N4)} can only admit reactions of these form. In each case, $i \neq j$ and $a_i$, $a_j > 0$. By $\cf{(\cdots)}$, we allow any combination of species except for $\cf{X}_i$ and $\cf{X}_j$.}
    \label{fig:SR-rxn-type}
\end{figure}

In what follows, we start with a reaction network satisfying conditions \textbf{(N2)}--\textbf{(N4)}. The only reactions allowed are those of the forms in \Cref{fig:SR-rxn-type}. By definition, the three left-most reactions in \Cref{fig:SR-rxn-type} are omitted from $V_R$ in the \SR. As a result, $V_R$ only contain R-nodes corresponding to reactions in the latter two columns in \Cref{fig:SR-rxn-type}.

\begin{rmk}
\label{rmk:fringe}
In how we defined the modified network in \Cref{def:modifiedMAS}, an awkward scenario can occur. If we start with the reaction network $\mc N$
    \tikzinline{1}{
        \node (xy) at (0,0) [left] {$\cf{X}+\cf{Y}$};
        \node (0) at (1.5,0) [right] {$\cf{0}$};
        \node (x) at (4,0) [left] {$\cf{X}$};
        \node (y) at (5.5,0) [right] {$\cf{Y}$};
        \draw [fwdrxn] (xy) -- (0) node [above] {\ratecnst{$\vphantom{\kk_1}$}};
        \draw [fwdrxn] (x) -- (y) ;
    }
then the modified network $\Gtilde$ consisting of  
    \tikzinline{1}{
        \node (xy) at (0,0.35) [left] {$\cf{X}$};
        \node (0) at (1.5,0.35) [right] {$\cf{Y}$};
        \node (a) at (0,-0.35) [left] {$\cf{Y}$};
        \node (b) at (1.5,-0.35) [right] {$\cf{X}$};
        \node (x) at (4,0) [left] {$\cf{X}$};
        \node (y) at (5.5,0) [right] {$\cf{Y}$};
        \draw [fwdrxn] (xy) -- (0)  node [above] {\ratecnst{$\vphantom{\kk_1}$}};
        \draw [fwdrxn] (x) -- (y);
        \draw [fwdrxn] (a) -- (b) ;
    }
awkwardly has a reaction repeated! According to definition, the set of R-nodes $V_R$ will only have one copy of the reaction $\cf{X} \to \cf{Y}$; however, for the rest of this section, \emph{we ensure that both R-nodes are included in the \SR}. See \Cref{fig:fringe} for the \SR s of $\mc N$ and $\Gtilde$. This is to ensure that cycles are not lost in going from the modified \SR\ to the original \SR\ under $\Phi$. 
\end{rmk}

\begin{rmk}
    It is worth remarking that in cases of repeated reactions in the modified network, one can check whether or not the modified Jacobian matrix $\tilde{\Jac}$ is the negative of a $P_0$-matrix using \Cref{cor:SR-cf2006}. The \SR\ where repeated reactions are not kept distinct is necessarily simpler. In the example of the above remark, the \SR\ consists of two S-nodes and one R-node (for the reaction $\cf{X} \RR \cf{Y}$), which has no cycle at all. So by \Cref{cor:SR-cf2006}, $-\tilde{\Jac}$ is a $P_0$-matrix and $\det(-\tilde{\Jac}) > 0$ for any vector of rate constants $\vv\kk > \vv 0$ and at any state $\xx > \vv 0$. 
\end{rmk}

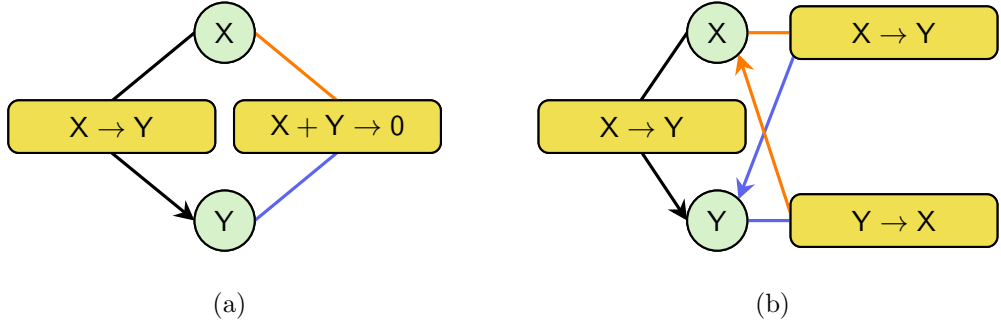
\begin{figure}[h!tbp]
\centering
\vspace{0.25cm}
\begin{subfigure}[b]{0.375\textwidth}
    \centering 
    \begin{tikzpicture}
\node at (0,-1.75) {};
    \node[Snode] (x) at (0,1.25) {$\cf{X}$};
    \node[Snode] (y) at (0,-1.25) {$\cf{Y}$};

    \node[Rnodelong] (r1) at (1.5,0) {$\cf{X}+\cf{Y} \to \cf{0}$};
    \node[Rnodelong] (r2) at (-1.5,0) {$\cf{X} \to \cf{Y}$};
    \draw [SR, highlightorange] (x.east)--([srNW] r1.north) ;
    \draw [SR, highlightblue] (y.east)--([srSW] r1.south) ;
    \draw [SR] (x.west)--([srNE] r2.north) ;
    \draw [DSR] (r2.south)--(y.west) ;

    \node[Snode] (x) at (0,1.25) {$\cf{X}$};
    \node[Snode] (y) at (0,-1.25) {$\cf{Y}$};
    \node[Rnodelong] (r1) at (1.5,0) {$\cf{X}+\cf{Y} \to \cf{0}$};
    \node[Rnodelong] (r2) at (-1.5,0) {$\cf{X} \to \cf{Y}$};
\end{tikzpicture}
    \caption{}
    \label{fig:fringe-a}
\end{subfigure}\hspace{0.05\textwidth}
\begin{subfigure}[b]{0.4\textwidth}
    \centering 
    \begin{tikzpicture}
\node at (0,-1.75) {};
    \node[Snode] (x) at (0,1.25) {$\cf{X}$};
    \node[Snode] (y) at (0,-1.25) {$\cf{Y}$};

    \node[Rnodelong] (r1a) at (2.35,1.25) {$\cf{X}\to\cf{Y}$};
    \node[Rnodelong] (r1b) at (2.35,-1.25) {$\cf{Y}\to\cf{X}$};
    \node[Rnodelong] (r2) at (-1,0) {$\cf{X} \to \cf{Y}$};
    \draw [SR, highlightorange] (x.east)--(r1a.west) ;
    \draw [SR, highlightblue] (y.east)--(r1b.west) ;
    \draw [SR] (x.west)--([srNE] r2.north) ;
    \draw [DSR] (r2.south)--(y.west) ;

    \draw [DSR, highlightblue] ([srSW] r1a.south west)--(y.north east) ;
    \draw [DSR, highlightorange] ([srNW] r1b.west)--(x.south east) ;
    
    \node[Snode] (x) at (0,1.25) {$\cf{X}$};
    \node[Snode] (y) at (0,-1.25) {$\cf{Y}$};
   \node[Rnodelong] (r1a) at (2.35,1.25) {$\cf{X}\to\cf{Y}$};
    \node[Rnodelong] (r1b) at (2.35,-1.25) {$\cf{Y}\to\cf{X}$};
    \node[Rnodelong] (r2) at (-1,0) {$\cf{X} \to \cf{Y}$};
\end{tikzpicture}
    \caption{}
    \label{fig:fringe-b}
\end{subfigure}
    \caption[\SR s of original and modified networks (a warning)]{The \SR s of (a) the original network and (b) the modified network from \Cref{rmk:fringe}. Duplicated R-nodes \emph{can} arise but should be kept for purpose of proof in \Cref{thm:SRmodified}.}
    \label{fig:fringe}
\end{figure}

We now define a map $\Phi$ from the vertices of the modified \SR\ to those of the original, smaller \SR, and we show that this map induces a surjective graph homomorphism between the \SR s. 

\begin{defn}
\label{def:SR-hom}
    Let $\DSR = (V_S, V_R, E, \sigma)$ be the \SR\ of a reaction network satisfying \textbf{(N2)}--\textbf{(N4)}, and $\tilde{\DSR} = (V_S, \tilde{V}_R, \tilde{E}, \tilde{\sigma})$ the \SR\ of the modified network as constructed in \Cref{sec:delay-modified}. Define a map $\Phi \colon V_S \cup \tilde{V}_R \to V_S \cup V_R$ as follows.
    \begin{itemize}
        \item Let $\Phi(\cf{X}) = \cf{X}$ for any $\cf{X} \in V_S$. 
        \item For any $\tilde{\cf{R}} \in \tilde{V}_R$ that does \emph{not} come from a bispecies reaction, there is a unique $\cf{R} \in V_R$ associated to it by the construction in \Cref{def:modifiedMAS}; let $\Phi(\tilde{\cf{R}}) = \cf{R}$.
        \item Any bispecies R-node $\cf{R} \in V_R$ is naturally associated to two R-nodes $\tilde{\cf{R}}_1$, $\tilde{\cf{R}}_2 \in \tilde{V}_R$, so let $\Phi(\tilde{\cf{R}}_1) = \Phi(\tilde{\cf{R}}_2) = \cf{R}$.
    \end{itemize} 
\end{defn}

In general, a \emph{homomorphism $F$ of directed graphs $G$ and $H$} is a map defined on the vertices $V_G \to V_H$ that preserves edges (adjacency and orientation), i.e., if $(u,v) \in E_G$ is a directed edge, then the directed edge $(F(u), F(v))$ lies in $E_H$~\cite{GraphHom}. Since some edges in \SR\ are drawn without a direction (i.e., it is compatible with both orientations), in order for $\Phi$ to be a graph homomorphism, we require that the image of an edge to admit the same orientation as the input. In other words, the image of a directed edge under $\Phi$ is either a directed edge with compatible orientation, or is an undirected edge.

Moreover, graph homomorphisms $F$ is said to be \emph{vertex-surjective} if the map $F \colon V_G \to V_H$ is surjective; it is said to be \emph{edge-surjective} if the map $F \colon E_G \to E_H$ is surjective. If $F$ is both vertex-surjective and edge-surjective, then $F$ is said to be a \emph{surjective} graph homomorphism~\cite{GraphHom}.   

\begin{prop}
\label{prop:phi-hom}
    The map $\Phi$ induces a surjective graph homomorphism from $\tilde{\DSR}$ to $\DSR$. 
\end{prop}
\begin{proof}
First $\Phi$ induces a map on the edges of the \SR\ in the obvious way: if $(\cf{X}, \tilde{\cf{R}})$ is an edge in $\tilde{E}$, then we let its image be $(\Phi(\cf{X}), \Phi(\tilde{\cf{R}})) = (\cf{X}, \Phi(\tilde{\cf{R}}))$. We need to show that $\Phi$ is well-defined and preserves both adjacency and orientation. 

There are two types of edges to consider. First, consider an edge $(\cf{X}, \tilde{\cf{R}}) \in \tilde{E}$ where $\Phi(\tilde{\cf{R}})$ is \emph{not} a bispecies R-node. The map $\Phi$ acts as the identity on these vertices. If the edge is undirected, either the reaction $\Phi(\tilde{\cf{R}})$ is reversible, or $\cf{X}$ is a reactant species in the irreversible reaction $\Phi(\tilde{R})$. Thus $(\cf{X}, \Phi(\tilde{\cf{R}})) \in E$ is undirected. If the edge is directed, necessarily the reaction is irreversible and $\cf{X}$ is a product species; hence the resulting edge is also directed. 

Now consider an edge $(\cf{X}, \tilde{\cf{R}}) \in \tilde{E}$ where $\Phi(\tilde{\cf{R}})$ is a bispecies R-node. If the edge is undirected, then $\cf{X}$ is a reactant species in the bispecies reaction. Hence $(\cf{X}, \Phi(\tilde{\cf{R}})) \in E$ is also undirected. However, if the edge in $\tilde{\DSR}$ is directed, the species $\cf{X}$ is either a product in the bispecies reaction (in which case $(\cf{X}, \Phi(\tilde{\cf{R}}))$ is also directed), or $\cf{X}$ is a reactant in the bispecies reaction (in which case $(\cf{X}, \Phi(\tilde{\cf{R}}))$ is undirected). Regardless, the directions of the edges are consistent under $\Phi$. Therefore, $\Phi$ is a graph homomorphism. By an abuse of notation, we let $\Phi \colon \tilde{\DSR} \to \DSR$, and allow $\Phi$ to act on vertices as well as edges of $\tilde{\DSR}$.

Lastly, we prove that $\Phi$ is both vertex-surjective and edge-surjective. Clearly every S-node in $\DSR$ is covered by the image of $\Phi$. Every R-node in $\DSR$ is representative of a reaction in the original network, which either is copied to or gives rise to two R-nodes in the modified network by \Cref{def:modifiedMAS}. This relationship between reactions is captured by $\Phi$ between the respective R-nodes in the \SR. In particular, $\Phi$ is vertex-surjective.  

Let $(\cf{X}, \cf{R})$ be any edge in $\DSR$. If $\cf{R}$ is a bispecies reaction, then in the modified network, there exists a reaction $\tilde{\cf{R}}$ for which $\cf{X}$ is a product species, i.e., $(\cf{X}, \Phi(\tilde{\cf{R}})) = (\cf{X}, \cf{R})$.  If $\cf{R}$ is \emph{not} a bispecies reaction, then there is a reaction $\tilde{R}$ that is an exact copy in the modified graph, so $(\cf{X}, \Phi(\tilde{\cf{R}})) = (\cf{X}, \cf{R})$. Thus, $\Phi$ is edge-surjective, and we conclude that $\Phi$ is a surjective graph homomorphism. 
\end{proof}

While $\Phi$ is surjective, in general it is not injective. However, in \Cref{lem:SR-hom-cycle} we show that $\Phi$, when restricted to a cycle $\tilde{C}$, is both vertex-injective and edge-injective whenever the image is a cycle. In such cases, we say that \df{$\Phi$ acts injectively on $\tilde{C}$}. For now, we start with the observation that $\Phi$ preserves the stoichiometric coefficients in a \SR.

\begin{lem}
\label{lem:SR-hom-label}
    The map $\Phi$ preserves stoichiometric coefficients, i.e., $\sigma \circ \Phi = \tilde{\sigma}$. 
\end{lem}
\begin{proof}
This follows from the construction in \Cref{def:modifiedMAS}; reactant species are always moved together with their stoichiometric coefficients.   
\end{proof}

In what follows, we frequently refer to the image of a subgraph under $\Phi$. If $\tilde{C}$ is a cycle in $\tilde{\DSR}$, we let $\Phi(\tilde{C})$ denote the subgraph in $\DSR$. More precisely, the set of vertices of $\Phi(\tilde{C})$ is $\{v \in V_S \cup V_R \colon \text{there is a vertex } \tilde{v} \text{ in } \tilde{C} \text{ such that } \Phi(\tilde{v}) = v  \}$. Similarly, the set of edges in $\Phi(\tilde{C})$ is the set of which there is an edge (with the appropriate orienation) in $\tilde{C}$ that gets mapped to it.

\begin{lem}
\label{rmk:cpair-scycle}
    Let $\tilde{C} \subseteq \tilde{\DSR}$ be a cycle that gets mapped to a c-pair under $\Phi$. The edges adjacent to a S-node in $\tilde{C}$ shares the same stoichiometric coefficient. In particular, $\tilde{C}$ is a s-cycle.  
\end{lem}
\begin{proof}
Proof of the lemma follows immediately from \Cref{fig:SRmodified-local-b}, where the coloured edges are representative of a cycle $\tilde{C}$ that gets mapped to a c-pair under $\Phi$. 
\end{proof}

In the following lemma, we show that if no cycle in $\DSR$ contains a bispecies production edge, then there is a one-to-one correspondence between the cycles in $\tilde{\DSR}$ and the set of cycles and c-pairs in $\DSR$. By ruling out bispecies production edges, the R-nodes in $\tilde{\DSR}$ associated to a bispecies reaction can only be adjacent to its reactant S-nodes in any cycle.

\begin{lem}
\label{lem:SR-hom-cycle}
    Suppose that no cycle in $\DSR$ contains a bispecies production edge. 
    \begin{enumerate}[label={\textup{(\alph*)}}]
        \item  If $\tilde{C} \subseteq \tilde{\DSR}$ is a cycle, then either $\Phi(\tilde{C})$ is a c-pair or $\Phi$ is injective on $\tilde{C}$.
        \item  For any cycle $C \subseteq \DSR$, there exists a unique cycle $\tilde{C} \subseteq \tilde{\DSR}$ such that $\Phi(\tilde{C}) = C$.
    \end{enumerate} 
     There is a one-to-one correspondence between the set of cycles in $\tilde{\DSR}$, and the set of cycles and c-pairs in $\DSR$. 
\end{lem}
\begin{proof}
\begin{enumerate}[label={\textup{(\alph*)}}]
\item 
    Restrict the map $\Phi$ to the cycle $\tilde{C} \subseteq \tilde{\DSR}$. If $\Phi$ is vertex-injective on $\tilde{C}$, then $\Phi(\tilde{C})$ is a cycle. However, if $\Phi$ is not vertex-injective, then there exist R-nodes $\tilde{\cf{R}}_i \neq \tilde{\cf{R}}_j$ in $\tilde{C}$ such that $\cf{R} = \Phi(\tilde{\cf{R}}_i) = \Phi(\tilde{\cf{R}}_j)$ is a bispecies R-node, with reactants $\cf{X}_i$ and $\cf{X}_j$. In $\tilde{\DSR}$, there is exactly one undirected edge adjacent to $\tilde{\cf{R}}_i$, while all other edges are outgoing directed edges. Without loss of generality, let $(\cf{X}_i, \tilde{\cf{R}}_i)$ be this undirected edge. In particular, in order for $\tilde{\cf{R}}_i$ to be in $\tilde{C}$, necessarily $(\cf{X}_i, \tilde{\cf{R}}_i)$ is in $\tilde{C}$. By a similar argument, the undirected edge $(\cf{X}_j, \tilde{\cf{R}}_j)$ must also be in $\tilde{C}$. 
    
    Next, we claim that in the cycle $\tilde{C}$ is the directed edge from $\tilde{\cf{R}}_i$ to $\cf{X}_j$ because no cycle in $\DSR$ contains a bispecies production edge. Suppose for a contradiction that $(\cf{X}_j, \tilde{\cf{R}}_i)$ is not an edge in $\tilde{C}$. Then $\tilde{C} = \left< \cf{X}_i, \tilde{\cf{R}}_i, \cf{X}_k, \cdots \right>$, where $\cf{X}_k$ is a product in $\cf{R}$, and the edge from $\cf{R}_i$ to $\cf{X}_k$ is a \emph{directed} edge. Thus, within $\Phi(\tilde{C})$ there is a cycle that contains the path $\cf{X}_i$, $\cf{R}$, and $\cf{X}_k$, i.e., there is a cycle with a bispecies production edge. Thus it must be the case that in the cycle $\tilde{C}$ is the edge from $\tilde{\cf{R}}_i$ to $\cf{X}_j$ and not to some other species. 
    
    An analogous argument as above imply that $\tilde{C} = \left< \cf{X}_i, \tilde{\cf{R}}_i, \cf{X}_j, \tilde{\cf{R}}_j\right>$, so $\Phi(\tilde{C})$ is a c-pair.

    
\item     
    Conversely, let $C\subseteq \DSR$ be a cycle. If $C$ does not contain any bispecies R-nodes, then because $\Phi$ acted bijectively on the cycle, $\Phi^{-1}(C) \subseteq \tilde{\DSR}$ is the unique cycle mapping to $C$. Suppose however, that $C$ contains a bispecies R-node; we claim that there is a unique cycle in $\tilde{\DSR}$ that gets mapped to $C$, as illustrated in \Cref{fig:lem:SR-hom-cycle}. For now, assume there is exactly one bispecies R-node $\cf{R}$. Let the cycle $C$ be $\left<v_\ell = v_0, v_1, v_2, \ldots, v_{\ell-1}\right>$, where $v_1 = \cf{R}$ is the bispecies R-node, with reactant species S-nodes $v_0 = \cf{X}_i$ and $v_2 = \cf{X}_j$. For $k \neq 1$, the vertex $v_k$ has a unique preimage under $\Phi$. However, the preimage of $v_1$ consists of two R-nodes: $\tilde{\cf{R}}_i$ for the modified reaction where $\cf{X}_i$ is the reactant, and $\tilde{\cf{R}}_j$ for the reaction where $\cf{X}_j$ is the reactant. Then $\Phi$ maps the cycle $\left<v_\ell = v_0, \tilde{\cf{R}}_i, v_2, \ldots, v_{\ell-1}\right> \subseteq \tilde{\DSR}$ uniquely to $C$. 

    If $C$ contains multiple bispecies R-node, a similar argument can be made locally at each bispecies R-node. Suppose $v_k = \cf{R}$ is a bispecies R-node, then $v_{k-1}$ and $v_{k+1}$ are S-nodes corresponding to the reactants of $\cf{R}$. Say $v_{k-1} = \cf{X}_i$ and $v_{k+1} = \cf{X}_j$. In the \SR\ $\DSR$ of the modified network, $\cf{R}$ is associated to two R-nodes: $\tilde{\cf{R}}_i$ with reactant $\cf{X}_i$ and $\tilde{\cf{R}}_j$ another with reactant $\cf{X}_j$. Then for the segment $\left< v_{k-1}, v_k, v_{k+1}\right>$ of the cycle, choose $\left< v_{k-1} = \cf{X}_i, \tilde{\cf{R}}_i, v_{k+1}=\cf{X}_j\right>$ as its preimage under $\Phi$. 
\end{enumerate}
\end{proof}

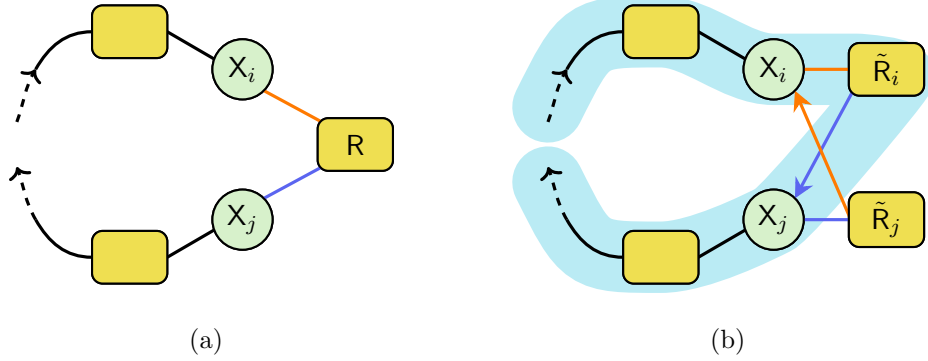
\begin{figure}[h!tbp]
\centering
\begin{subfigure}[t]{0.37\textwidth}
    \centering 
    \begin{tikzpicture}
    \node at (0,-2) {};
    \node[Snode] (x) at (0,1) {$\cf{X}_i$};
    \node[Snode] (y) at (0,-1) {$\cf{X}_j$};
    \node[Rnode] (r1) at (1.5,0) {$\cf{R}$};
    \draw [SR, highlightorange] ([yshift=0cm] x.south east)--([srNW] r1.north west) ;
    \draw [SR, highlightblue] ([yshift=-0cm] y.north east)--([srSW] r1.south west) ;
    
    \node[Rnode] (l1) at (-1.5,1.5) {};
    \node[Rnode] (3) at (-1.5,-1.5) {};
    
    \draw [SR] ([yshift=-0.2cm] x.north west)--(l1.east) ;
    \draw [SR] ([yshift=0.2cm] y.south west)--(3.east) ;
    
    \draw [SR] (3.west) to [out=180, in=-60] (-2.75,-1);
    \draw [SR, dashed, ->] (-2.75,-1) to [out=120, in=-70] (-3,-0.3);
    
    \draw [SR, dashed, ->] (-3,0.3) to [out=70, in=-120] (-2.75,1);
    \draw [SR] (-2.75,01) to [out=60, in=-180] (l1.west);

    \node[Snode] (x) at (0,1) {$\cf{X}_i$};
    \node[Snode] (y) at (0,-1) {$\cf{X}_j$};
    \node[Rnode] (r1) at (1.5,0) {$\cf{R}$};
    \node[Rnode] (l1) at (-1.5,1.5) {};
    \node[Rnode] (3) at (-1.5,-1.5) {};
\end{tikzpicture}
    \caption{}
    \label{fig:lem:SR-hom-cycle-a}
\end{subfigure}\hspace{0.05\textwidth}
\begin{subfigure}[t]{0.37\textwidth}
    \centering 
    \begin{tikzpicture}
\draw [cycopen, highlightbabyblue] (-3,0.5) .. controls (-2.5,1.5) .. (-1.5,1.5)
    .. controls (-1,1.5) and (-0.5,1) ..  (0,1)
    .. controls (1,1).. (1.1,1)
    .. controls (1.2,1) and (1.65,0.95).. (1.65,0.9)
    .. controls (1.65,0.85) and (0.2,-0.9) .. (0,-1)
    .. controls (-0.2,-1.1) and (-1,-1.5) .. (-1.5,-1.5)
    .. controls (-2.5,-1.5) .. (-3,-0.5)
    ;

    \node at (0,-2) {};
    \node[Snode] (x) at (0,1) {$\cf{X}_i$};
    \node[Snode] (y) at (0,-1) {$\cf{X}_j$};
    \node[Rnode] (r1) at (1.5,1) {$\tilde{\cf{R}}_i$};
    \node[Rnode] (r2) at (1.5,-1) {$\tilde{\cf{R}}_j$};
    \draw [SR, highlightorange] (x.east)--(r1.west) ;
    \draw [SR, highlightblue] (y.east)--(r2.west) ;
    
    \draw [DSR, highlightblue] ([srSW] r1.south west)--(y.north east) ;
    \draw [DSR, highlightorange] ([srNW] r2.west)--(x.south east) ;
    \node[Rnode] (l1) at (-1.5,1.5) {$v_{\ell-1}$};
    \node[Rnode] (3) at (-1.5,-1.5) {$v_{3}$};
    
    \draw [SR] ([yshift=-0.2cm] x.north west)--(l1.east) ;
    \draw [SR] ([yshift=0.2cm] y.south west)--(3.east) ;
    
    \draw [SR] (3.west) to [out=180, in=-60] (-2.75,-1);
    \draw [SR, dashed, ->] (-2.75,-1) to [out=120, in=-70] (-3,-0.3);
    
    \draw [SR, dashed, ->] (-3,0.3) to [out=70, in=-120] (-2.75,1);
    \draw [SR] (-2.75,01) to [out=60, in=-180] (l1.west);

    \node[Snode] (x) at (0,1) {$\cf{X}_i$};
    \node[Snode] (y) at (0,-1) {$\cf{X}_j$};
    \node[Rnode] (r1) at (1.5,1) {$\tilde{\cf{R}}_i$};
    \node[Rnode] (r2) at (1.5,-1) {$\tilde{\cf{R}}_j$};
    \node[Rnode] (l1) at (-1.5,1.5) {};
    \node[Rnode] (3) at (-1.5,-1.5) {};

\end{tikzpicture}
    \caption{}
    \label{fig:lem:SR-hom-cycle-b}
\end{subfigure}
    \caption[Correspondence of cycles in original and modified \SR s if no cycle contains bispecies production edge]{If no cycle in $\DSR$ contains a bispecies production edge, the preimage of a cycle in $\DSR$ is a unique cycle in $\tilde{\DSR}$. A cycle $C \subseteq \DSR$ containing a bispecies R-node (a) is uniquely mapped from the cycle in (b). The edges adjacent to the R-node are coloured to emphasize the pairing in the \SR s, and arrows indicate direction of the cycle.}
    \label{fig:lem:SR-hom-cycle}
\end{figure}

We are ready to approach the proof of \Cref{thm:SRmodified}. When no cycle contains a bispecies production edge,  \Cref{prop:SR-hom-s-cycle} relates the s-cycles of the \SR s, while \Cref{prop:SR-hom-SRintxn} relates the S-to-R intersections. These two propositions together imply \Cref{thm:SRmodified}. 

Note that in the following proposition, it is not necessary to assume that $\DSR$ does not have a cycle with a bispecies product edge. However, with this assumption, the proof greatly simplifies.

\begin{prop}
\label{prop:SR-hom-s-cycle}
    Assume that no cycle in $\DSR$ contains a bispecies production edge.  Then all cycles in $\DSR$ are s-cycles if and only if all cycles in $\tilde{\DSR}$ are s-cycles. 
\end{prop}
\begin{proof}
    Suppose all cycles in $\DSR$ are s-cycles. Let $\tilde{C} \subseteq \tilde{\DSR}$ be a cycle. \Cref{lem:SR-hom-cycle} implies that $\Phi(\tilde{C})$ is either a cycle or a c-pair. Recall that $\Phi$ preserves the stoichiometric coefficients. Thus $\tilde{C}$ is a s-cycle by assumption in the former case and by \Cref{rmk:cpair-scycle} in the latter. 
    
    Conversely, suppose all cycles in $\tilde{\DSR}$ are s-cycles, and let $C \subseteq \DSR$ be any cycle. By \Cref{lem:SR-hom-cycle}, $\Phi^{-1}(C)$ is a cycle, thus a s-cycle by assumption. Hence, its image $C$ is also a s-cycle.
\end{proof}

Finally, we reached our last proposition, and the subtle connection between the \SR s of a network and its modified version. \Cref{fig:SRIntxn-counterex} shows that if $\DSR$ has a cycle with a bispecies production edge, then there is a S-to-R intersection in  $\tilde{\DSR}$. 

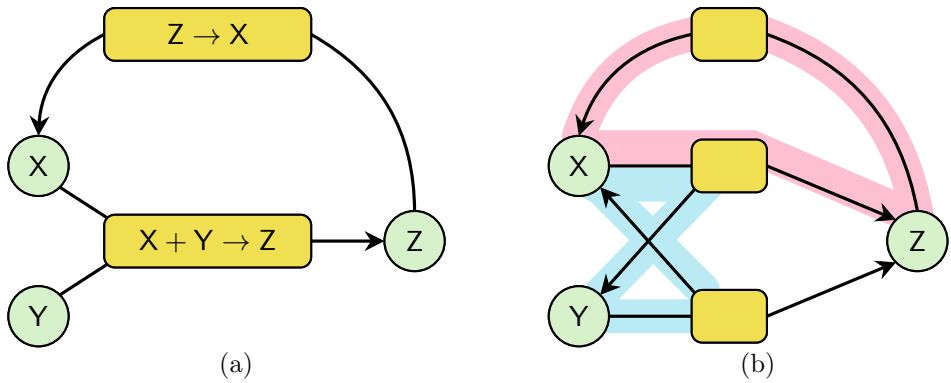
\begin{figure}[h!tbp]
\centering
\begin{subfigure}[t]{0.37\textwidth}
    \centering 
    \begin{tikzpicture}
    \node at (0,3.3) {};
    \node[Snode] (x) at (0.25,1) {$\cf{X}$};
    \node[Snode] (y) at (0.25,-1) {$\cf{Y}$};
    \node[Snode] (z) at (5.25,0) {$\cf{Z}$};
    \node[Rnodelong] (r1) at (2.5,0) {$\cf{X}+\cf{Y} \to \cf{Z}$};
    \node[Rnodelong] (r2) at (2.5,2.75) {$\cf{Z} \to \cf{X}$};
    \draw [SR] (x.south east)--([srNW] r1.north west) node [midway, above] {\footnotesize };
    \draw [SR] (y.north east)--([srSW] r1.south west) node [midway, below] {\footnotesize };
    \draw [DSR] (r1.east)--(z.west) node [midway, above] {\footnotesize };
    \draw [SR] (z.north) to [bend right] node[midway, above] {\footnotesize } (r2.east) ;
    \draw [DSR] (r2.west) to [bend right] node [midway, above] {\footnotesize } (x.north);
    
   \node[Snode] (x) at (0.25,1) {$\cf{X}$};
    \node[Snode] (y) at (0.25,-1) {$\cf{Y}$};
    \node[Snode] (z) at (5.25,0) {$\cf{Z}$};
    \node[Rnodelong] (r1) at (2.5,0) {$\cf{X}+\cf{Y} \to \cf{Z}$};
    \node[Rnodelong] (r2) at (2.5,2.75) {$\cf{Z} \to \cf{X}$};
\end{tikzpicture}
    \vspace{-0.25cm}
    \caption{}
    \label{fig:SRIntxn-counterex-a}
\end{subfigure}\hspace{0.05\textwidth}
\begin{subfigure}[t]{0.37\textwidth}
    \centering 
    \begin{tikzpicture}
\node at (0,3.3) {};
    \node[Snode, opacity=0] (x) at (0.25,1) {$\cf{X}$};
    \node[Snode, opacity=0] (y) at (0.25,-1) {$\cf{Y}$};
    \node[Snode, opacity=0] (z) at (5.25,0) {};
    \node[Rnodelong,  opacity=0] (r1) at (2.5,0) {};
    \node[Rnodelong, opacity=0] (r2) at (2.5,2.75) {};
    \draw [SR, opacity=0] (x.east)--([srNW] r1.north west) node [midway, above] {\footnotesize };
    \draw [SR, opacity=0] (y.east)--([srSW] r1.south west) node [midway, below] {\footnotesize };
    \draw [DSR, opacity=0] (r1.east)--(z.west) node [midway, above] {\footnotesize };
    \draw [SR, opacity=0] (z.north) to [bend right] node[midway, above] {\footnotesize } (r2.east) ;
    \draw [DSR, opacity=0] (r2.west) to [bend right] node [midway, above] {\footnotesize } (x.north);

   \node[Snode] (x) at (0.5,1) {};
    \node[Snode] (y) at (0.5,-1) {};
    \node[Snode] (z) at (5,0) {};
    \node[Rnode] (r1a) at (2.5,1) {};
    \node[Rnode] (r1b) at (2.5,-1) {};
    \node[Rnode] (r2) at (2.5,2.75) {};
    \draw [cycopenhalf, highlightbabypink]  ([yshift=0.25cm,xshift=-0.2cm]x.east)--([yshift=0.25cm, xshift=-0.2cm]r1a.east)--([yshift=-0.1cm]z.north) to [bend right] (r2.east)-- (r2.west) to [bend right] (x.north); 
    \draw [cycopenhalf, highlightbabyblue]  ([yshift=-0.25cm,xshift=-0.2cm]x.east)--([yshift=-0.25cm, xshift=0.2cm]r1a.west) -- (y)--(r1b)--(x); 
    \draw [SR] (x.east)--(r1a.west) node [midway, above] {\footnotesize };
        \draw [DSR] ([srNW] r1b.north west)--(x.south east) node [midway, above] {\footnotesize };
    \draw [SR] (y.east)--(r1b.west) node [midway, below] {\footnotesize };
        \draw [DSR] ([srSW] r1a.south west)--(y.north east) node [midway, above] {\footnotesize };
    \draw [DSR] (r1a.east)--(z.north west) node [midway, above] {\footnotesize };
    \draw [DSR] (r1b.east)--(z.south west) node [midway, above] {\footnotesize };
    \draw [SR] (z.north) to [bend right] node[midway, above] {\footnotesize } (r2.east) ;
    \draw [DSR] (r2.west) to [bend right] node [midway, above] {\footnotesize } (x.north);

    \node[Snode] (x) at (0.5,1) {\cf{X}};
    \node[Snode] (y) at (0.5,-1) {\cf{Y}};
    \node[Snode] (z) at (5,0) {\cf{Z}};
    \node[Rnode] (r1a) at (2.5,1) {};
    \node[Rnode] (r1b) at (2.5,-1) {};
    \node[Rnode] (r2) at (2.5,2.75) {};
\end{tikzpicture}
    \vspace{-0.25cm}
    \caption{ }
    \label{fig:SRIntxn-counterex-b}
\end{subfigure}
    \caption[S-to-R intersection exists if a cycle contains a bispecies production edge]{(a) The \SR\ of a network that has a cycle with a bispecies production edge. (b) The \SR\ of the modified network, which contains a S-to-R intersection. Two cycles are highlighted whose intersection is one edge.}
    \label{fig:SRIntxn-counterex}
\end{figure}

\begin{prop}
\label{prop:SR-hom-SRintxn}
    There is a S-to-R intersection in $\tilde{\DSR}$ if and only if there is either
    \begin{enumerate}[label={\textup{(\alph*)}}]
        \item 
             S-to-R intersection in $\DSR$, or 
        \item 
            a cycle in $\DSR$ containing a bispecies production edge.
    \end{enumerate}
\end{prop}
\begin{proof}
    First we prove that if $\DSR$ has no S-to-R intersection and no cycles with a bispecies production edge, then $\tilde{\DSR}$ cannot have S-to-R intersection. Let $\tilde{C}_1 \neq \tilde{C}_2$ be two cycles in $\tilde{\DSR}$ such that $\tilde{C}_2 \cap \tilde{C}_2$ contains some R-node and S-node in a connected component; we need to show that $\tilde{C}_1 \cap \tilde{C}_2$ is \emph{not} a S-to-R intersection. Without loss of generality, \Cref{lem:SR-hom-cycle} implies we have three cases (\Cref{fig:SRIntxn-pf}) to handle separately:
    \begin{enumerate}
    \item $\Phi(\tilde{C}_1)$ and $\Phi(\tilde{C}_2)$ are c-pairs;
    \item $\Phi(\tilde{C}_1)$ is a cycle while $\Phi(\tilde{C}_2)$ is a c-pair;
    \item $\Phi(\tilde{C}_1)$ and $\Phi(\tilde{C}_2)$ are both cycles. 
    \end{enumerate}

    In the case where $\Phi(\tilde{C}_1)$ and $\Phi(\tilde{C}_2)$ are c-pairs (\Cref{fig:SRIntxn-pf-a}), they must share the same R-node, which is a bispecies R-node in $\DSR$. Since a bispecies R-node has exactly two incoming (undirected in this case) edges, $\tilde{C}_1 = \tilde{C}_2$, which contradicts our assumption.
    
    In the next case, the unique R-node of the c-pair $\Phi(\tilde{C}_2)$ is a bispecies R-node in the cycle $\Phi(\tilde{C}_1)$. Because the two edges in $\Phi(\tilde{C}_2)$ form part of the cycle, $\Phi(\tilde{C}_2)$ has the form of the partially shown cycle in \Cref{fig:SRIntxn-pf-b}. The intersection $\tilde{C}_1 \cap \tilde{C}_2$ consists of exactly two edges in $\tilde{\DSR}$ that get mapped to the c-pair. The intersection in $\tilde{\DSR}$ is similar to the segment $\left<\cf{X}_i, \tilde{\cf{R}}_i, \cf{X}_j\right>$ in \Cref{fig:lem:SR-hom-cycle}. Thus, $\tilde{C}_1 \cap \tilde{C}_2$ starts and ends at S-nodes, i.e., it is not a S-to-R intersection. 
    
    In the final case, $\Phi(\tilde{C}_1)$ and $\Phi(\tilde{C}_2)$ are cycles. Since no cycle in $\DSR$ contains a bispecies production edge, $\Phi$ acts injectively on the cycles $\tilde{C}_1$ and $\tilde{C}_2$, and $\Phi(\tilde{C}_1 \cap \tilde{C}_2) = \Phi(\tilde{C}_1) \cap \Phi(\tilde{C}_2)$. Because $\DSR$ does not have any S-to-R intersection, $\tilde{C}_1 \cap \tilde{C}_2$ is not a S-to-R intersection.

\begin{figure}[h!tbp]
\centering
\begin{subfigure}[b]{0.25\textwidth}
    \centering 
    \begin{tikzpicture}
    \node at (0,-1.5) {};
    \node at (0,1) {};
    
    \node[Snode] (x) at (-1,1) {};
    \node[Snode] (y) at (-1,-1) {};
    \node[Rnode] (r) at (0,0) {};
    \draw [SR, highlightblue] (x.south east)--([srNW]r.north west) ;
    \draw [SR, highlightblue] (y.north east)--([srSW]r.south west) ;
    
    \begin{scope}[transform canvas={xshift=0.06cm, yshift=0.055cm}]
    \draw [SR] (x.south east)--([srNW]r.north west) ;
    \end{scope}
    \begin{scope}[transform canvas={xshift=0.06cm, yshift=-0.055cm}]
    \draw [SR] (y.north east)--([srSW]r.south west) ;
    \end{scope}

   \node[Snode] (x) at (-1,1) {};
    \node[Snode] (y) at (-1,-1) {};
    \node[Rnode] (r) at (0,0) {};
\end{tikzpicture}
    \caption{}
    \label{fig:SRIntxn-pf-a}
\end{subfigure}\hspace{0.05\textwidth}
\begin{subfigure}[b]{0.25\textwidth}
    \centering 
    \begin{tikzpicture}
    \node at (0,-1.5) {};
    \node at (0,1) {};
    
    \node[Snode] (x) at (-1,1) {};
    \node[Snode] (y) at (-1,-1) {};
    \node[Rnode] (r) at (0,0) {};
    \draw [SR, highlightblue] (x.south east)--([srNW]r.north west) ;
    \draw [SR, highlightblue] (y.north east)--([srSW]r.south west) ;
    
    \draw [SR] (y.south west) to [out=220, in=-60, looseness=1] (-2,-1);
     \draw [SR,->, dashed] (-2,-1) to [out=120, in=-80] (-2.25,-0.25);
    \draw [SR,->,dashed] (-2.25,0.25)  to [out=80, in=-120] (-2,1);
    \draw [SR] (-2,1) to [out=60, in=-220] (x.north west);
    
    \begin{scope}[transform canvas={xshift=0.06cm, yshift=0.055cm}]
    \draw [SR] (x.south east)--([srNW]r.north west) ;
    \end{scope}
    \begin{scope}[transform canvas={xshift=0.06cm, yshift=-0.055cm}]
    \draw [SR] (y.north east)--([srSW]r.south west) ;
    \end{scope}

    \node[Snode] (x) at (-1,1) {};
    \node[Snode] (y) at (-1,-1) {};
    \node[Rnode] (r) at (0,0) {};
\end{tikzpicture}
    \caption{}
    \label{fig:SRIntxn-pf-b}
\end{subfigure}\hspace{0.05\textwidth}
\begin{subfigure}[b]{0.25\textwidth}
    \centering 
    \begin{tikzpicture}
    \node at (0,-1.5) {};
    \node at (0,1) {};
    
    \node[Snode] (x) at (-1.25,0) {};
    \node[Snode] (y) at (1.25,0) {};
    
    \draw [SR] (x.east)--(-0.75,0);
    \draw [SR,->,dashed] (-0.75,0)--(-0.25,0);
    \draw [SR,->,dashed] (0.25,0)--(0.65,0);
    \draw [SR] (0.65,0)--(y.west);
    \begin{scope}[transform canvas={yshift=-0.07cm}]
    \draw [SR, highlightblue] (x.east)--(-0.75,0);
    \draw [SR,dashed, highlightblue] (-0.75,0)--(-0.25,0);
    \draw [SR,dashed, highlightblue] (0.25,0)--(0.65,0);
    \draw [SR, highlightblue] (0.65,0)--(y.west);
    \end{scope}
    
    \draw [SR] (y.north east) to [out=60, in=-30] (0.9,1.25);
    \draw [SR,->,dashed] (0.9,1.25) to [out=150, in=0] (0.2,1.4);
    \draw [SR,->,dashed] (-0.2,1.4) to [out=180, in=30] (-0.9,1.25);
    \draw [SR] (-0.9,1.25) to [out=210, in=130] (x.north west);
    
    \draw [SR, highlightblue] (y.south east) to [out=-60, in=30] (0.9,-1.25);
    \draw [SR,->,dashed, highlightblue] (0.9,-1.25) to [out=210, in=0] (0.2,-1.4);
    \draw [SR,->,dashed, highlightblue] (-0.2,-1.4) to [out=180, in=-30] (-0.9,-1.25);
    \draw [SR, highlightblue] (-0.9,-1.25) to [out=150, in=230] (x.south west);

\end{tikzpicture}
    \caption{}
    \label{fig:SRIntxn-pf-c}
\end{subfigure}
    \caption{Preimages in $\tilde{\DSR}$ of two cycles with at least a common S-node and a common R-node. The three cases to consider in the proof of \Cref{prop:SR-hom-SRintxn} are when (a) both $\Phi(\tilde{C}_i)$ are c-pairs, (b) one of them is a c-pair, and (c) neither is a c-pair. In (a) and (b), the intersection $\tilde{C}_1 \cap\tilde{C}_2$ is the preimage of the c-pair. In (c), because $\Phi(\tilde{C}_1)$ and $\Phi(\tilde{C}_2)$ are cycles on which $\Phi$ acts injectively, $\Phi(\tilde{C}_1 \cap \tilde{C}_2) = \Phi(\tilde{C}_1)\cap \Phi(\tilde{C}_2)$.}
    \label{fig:SRIntxn-pf}
\end{figure}
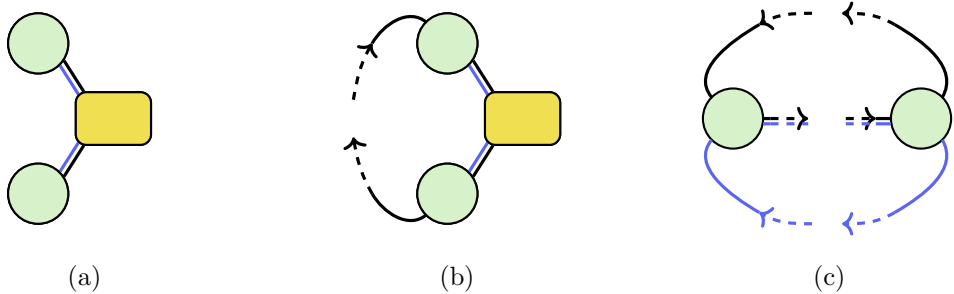


    Conversely, first suppose now that no cycle contains a bispecies production edge, but that there are two cycles $C_1 \neq C_2$ in $\DSR$ whose intersection $C_1 \cap C_2$ is a S-to-R intersection. By \Cref{lem:SR-hom-cycle}, let $\tilde{C}_1$ and $\tilde{C}_2$ be the unique cycles mapped to $C_1$ and $C_2$ respectively. Since $\Phi$ acts injectively on these cycles, we have $\Phi(\tilde{C}_1 \cap \tilde{C}_2) = C_1 \cap C_2$, a S-to-R intersection; therefore $\tilde{C}_1 \cap \tilde{C}_2$ is also a S-to-R intersection. 
    
    Now suppose that a cycle $C$ in $\DSR$ contains a bispecies production edge $(\cf{X}_k, \cf{R})$ where $\cf{R}$ is the bispecies R-node, with reactant S-nodes $\cf{X}_i$ and $\cf{X}_j$. \Cref{fig:SRIntxn-counterex} hints at the proof. Let $\tilde{C}^*$ be the s-cycle that is the preimage of the c-pair adjacent to $\cf{R}$, i.e., $\tilde{C}^* = \left< \cf{X}_i, \tilde{\cf{R}}_i, \cf{X}_j, \tilde{\cf{R}}_j\right>$. We will construct a cycle in $\tilde{\DSR}$ whose intersection with $\tilde{C}^*$ is a S-to-R intersection. Let $C = \left< v_\ell = v_0, v_1, \ldots, v_{\ell-1} \right>$, where $v_1 = \cf{R}$, $v_0 = \cf{X}_i$, and $v_2 = \cf{X}_k$. Then $\Phi$ maps the segment $\left< \cf{X}_i, \tilde{\cf{R}}_i, \cf{X}_k \right>$ to the $\left<v_0, v_1,v_2 \right>$ part of the cycle $C$. Note that the intersection of $\left< \cf{X}_i, \tilde{\cf{R}}_i, \cf{X}_k \right>$ with $\tilde{C}^*$ is $(\cf{X}_k, \tilde{\cf{R}}_i)$, a S-to-R intersection. More precisely, whenever $v_k$ is a bispecies R-node, and $v_{k-1} = \cf{X}_i$, choose $\tilde{\cf{R}}_i$ from the preimage of $v_k$. The result is a cycle in $\tilde{\DSR}$, whose intersection with $\tilde{C}^*$ is precisely $(\cf{X}_k, \tilde{\cf{R}}_i)$, a S-to-R intersection. 
\end{proof}

\subsection{Proof of \Cref{thm:SRmodified}}

Recall that we supposed $\mc N$ satisfy conditions \textbf{(N2)}--\textbf{(N4)};  \Cref{thm:SRmodified} claims that in the \SR\ $\tilde{\DSR}$ of the network $\tilde{\mc N}$, 
    \begin{enumerate}[topsep=-3pt]
    \item[\textup{(i)}] 
        all cycles are s-cycles, and 
    \item[\textup{(ii)}] 
        there is no S-to-R intersection,
    \end{enumerate}
if and only if in the \SR\ $\DSR$ of the network $\mc N$,
    \begin{enumerate}[topsep=-3pt]
    \item[\textup{(a)}]
        no cycle contains a bispecies production edge; 
    \item[\textup{(b)}] 
        all cycles are s-cycles, and 
    \item[\textup{(c)}] 
        there is no S-to-R intersection. 
    \end{enumerate}
\Cref{prop:SR-hom-SRintxn} claims \textup{(ii)} is equivalent to \textup{(a)} and \textup{(c)}. Assuming \text{(a)}, \Cref{prop:SR-hom-s-cycle} proves the equivalence of \textup{(i)} and \textup{(b)}. Clearly, \textup{(a)}--\textup{(c)} implies \textup{(i)}--\textup{(ii)}. Towards the other direction, \textup{(ii)} implies \textup{(a)} and \textup{(c)}. Hence, we may take \textup{(a)} as an assumption, and conclude \textup{(b)} from \textup{(i)}. \qed

\subsection{Summary}
\label{sec:implications}

\begin{figure}[h!]
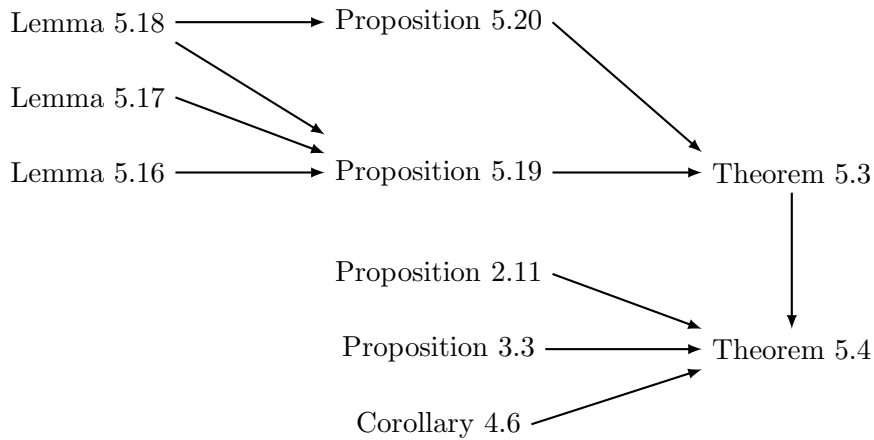

\centering
\tikzinline{1}{
\node at (0,2.5){}; 
\node (lem6) at (0,0) [left] {\Cref{lem:SR-hom-label}};

\node (lem7) at (0,1) [left] {\Cref{rmk:cpair-scycle}}; 

\node (lem8) at (0,2) [left] {\Cref{lem:SR-hom-cycle}};
\node (prop9) at (3.5,0) {\Cref{prop:SR-hom-s-cycle}};
    \draw [-latex, thick] (lem6.east)--(prop9.west);
    \draw [-latex, thick] (lem7.east)--([yshift=-0.05cm] prop9.north west);
    \draw [-latex, thick] (lem8.south east)--([yshift=0.2cm] prop9.north west);
    
\node  (prop10) at (3.5,2) {\Cref{prop:SR-hom-SRintxn}};
    \draw [-latex, thick] (lem8.east)--(prop10.west);
\node (thm2) at (7,0) [right] {\Cref{thm:SRmodified}};
    \draw [-latex, thick] (prop10.east)--(thm2.north west);
    \draw [-latex, thick] (prop9.east)--(thm2.west);
\node (thm3) at (7,-2.35) [right] {\Cref{thm:delay-SR}};
    \draw [-latex, thick] (thm2.south)--(thm3.north);

\node (cordelay) at (3.5,-1.35)  {\Cref{cor:delay-alg}};
\node (propmodJac) at (3.5,-2.35)  {\Cref{prop:modifiedsystem}};
\node (corcf06) at (3.5,-3.35)  {\Cref{cor:SR-cf2006}};
    \draw [-latex, thick] (cordelay.east)--(thm3.north west);
    \draw [-latex, thick] (propmodJac.east)--(thm3.west);
    \draw [-latex, thick] (corcf06.east)--(thm3.south west);
}
\caption{Logical implications of results}
\label{fig:implications}
\end{figure}

Results have become more intertwined in the latter portion of this section. We now summarize the various results, and the connections between them are shown in \Cref{fig:implications}.  

First, recall that \Cref{cor:delay-alg} cites algebraic conditions on the Jacobian and modified Jacobian matrices of a reaction network $\mc N$ that are sufficient to conclude delay stability of delay mass-action systems~\cite{CraciunMinchevaPanteaYu2020}. In particular, to conclude asymptotic stability of any equilibrium, we need the Jacobian matrix to be non-singular, and the negative of the modified Jacobian matrix $-\tilde{\Jac}$ to be a $P_0$-matrix whose diagonal is strictly positive. In \Cref{prop:modifiedsystem}, we constructed the modified reaction network $\tilde{\mc N}$ such that its Jacobian matrix, when evaluated at appropriate rate constants and concentrations, is exactly the modified Jacobian matrix of $\mc N$. Meanwhile, \Cref{cor:SR-cf2006} references known conditions on \SR\ of a reaction network sufficient for the negative of its Jacobian matrix to be a $P_0$-matrix~\cite{CraciunFeinberg2006,CraciunBanaji2010}. Hence, we may conclude that $\tilde{\Jac}$ has the desired properties for delay stability, provided the \SR\ of the modified network $\tilde{\mc N}$ satisfies the following conditions:
\begin{itemize}
    \item all cycles are s-cycles, and 
    \item there is no S-to-R intersection.
\end{itemize}

Next, recognizing that the modified network $\tilde{\mc N}$ is an artifact born because of the modified Jacobian matrix $\tilde{\Jac}$, we attempt to relate the \SR\ of the reaction network $\mc N$ and that of the modified network $\tilde{\mc N}$. Because of how the modified network is defined, the \SR\ $\DSR$ of $\mc N$ looks extremely similar to the \SR\ $\tilde{\DSR}$ of $\tilde{\mc N}$. Notably, the differences between the two \SR s are due to any reactions with multiple reactant species. For technical reasons, we restricted ourselves with networks $\mc N$ with four properties \textbf{(N1)}--\textbf{(N4)}. For the purpose of this discussion, note that we allow \emph{at most two} reactant species for each reaction; for reactions with exactly two reactant species --- the so-called \emph{bispecies reactions} --- these must be irreversible. The differences between the \SR s of $\mc N$ and $\tilde{\mc N}$ occur near any bispecies reaction R-nodes. 

A feature that we would like to avoid in the \SR\ of $\mc N$ is cycle with a bispecies production edge (see \Cref{def:bispeciesedge} or \Cref{fig:SRIntxn-bisp-prodt}). In \Cref{prop:SR-hom-SRintxn}, we prove that the \SR\ $\DSR$ of the network $\mc N$ has a S-to-R intersection if and only if the \SR\ $\tilde{\DSR}$ of the modified network $\tilde{\mc N}$ either has a S-to-R intersection or has a cycle containing a bispecies production edge. In \Cref{lem:SR-hom-cycle}, we showed that there is a one-to-one correspondence between the set of cycles in $\tilde{\DSR}$ and the set of cycles and c-pairs in $\DSR$.
\Cref{rmk:cpair-scycle} and  \Cref{lem:SR-hom-label} relate the stoichiometric coefficients of $\DSR$ and $\tilde{\DSR}$, essentially allowing us to conclude that all cycles in $\DSR$ are s-cycles if every cycle in $\tilde{\DSR}$ is a s-cycle in \Cref{prop:SR-hom-s-cycle}. In short, we conclude in \Cref{thm:delay-SR} that $\mc N$ is delay stable, i.e., any equilibrium is asymptotically stable independent of rate constants and delay parameters, if $\mc N$ satisfies conditions \textbf{(N1)}--\textbf{(N4)} and the \SR\ $\DSR$ of $\mc N$ satisfies the following conditions: 
\begin{itemize}
    \item no cycle contains a bispecies production edge; 
    \item all cycles are s-cycles, and 
    \item there is no S-to-R intersection.
\end{itemize}

\section*{Acknowledgements} 

G.C. was supported in part by the National Science Foundation [DMS--1412643, DMS--1816238], C.P. was supported in part by [NSF DMS--1517577], and P.Y. was supported in part by Natural Sciences and Engineering Research Council of Canada [PGS-D].

\bibliographystyle{siam}
\bibliography{cit}

\begin{thebibliography}{10}

\bibitem{Angeli.2013aa}
{\sc D.~Angeli, M.~Banaji, and C.~Pantea}, {\em Combinatorial approaches to
  {H}opf bifurcations in systems of interacting elements}, Communications in
  Mathematical Sciences, 12 (2014), pp.~1101 -- 1133.

\bibitem{Angeli.2010aa}
{\sc D.~Angeli, P.~De~Leenheer, and E.~Sontag}, {\em Graph-theoretic
  characterizations of monotonicity of chemical networks in reaction
  coordinates}, J Math Biol, 61 (2010), pp.~581--616.

\bibitem{Banaji2012}
{\sc M.~Banaji}, {\em Cycle Structure in {SR} and {DSR} Graphs: {I}mplications
  for Multiple Equilibria and Stable Oscillation in Chemical Reaction
  Networks}, vol.~6900 of Lecture Notes in Computer Science, Springer, 2012,
  pp.~1--22.

\bibitem{CraciunBanaji2009}
{\sc M.~Banaji and G.~Craciun}, {\em Graph-theoretic approaches to injectivity
  and multiple equilibria in systems of interacting elements}, Communications
  in Mathematical Sciences, 7 (2009), pp.~867--900.

\bibitem{CraciunBanaji2010}
\leavevmode\vrule height 2pt depth -1.6pt width 23pt, {\em Graph-theoretic
  criteria for injectivity and unique equilibria in general chemical reaction
  systems}, Advances in Applied Mathematics, 44 (2010), pp.~168--184.

\bibitem{BanajiRutherford2011}
{\sc M.~Banaji and C.~Rutherford}, {\em {$P$}-matrices and signed digraphs},
  Discrete Mathematics, 311 (2011), pp.~295--301.

\bibitem{BanksRobbinsSutton2013}
{\sc H.~T. Banks, D.~Robbins, and K.~L. Sutton}, {\em Theoretical foundations
  for traditional and generalized sensitivity functions for nonlinear delay
  differential equations}, Mathematical Biosciences and Engineering, 10 (2013),
  pp.~1301--1333.

\bibitem{Bodner2000}
{\sc M.~Bodnar}, {\em The nonnegativity of solutions of delay differential
  equations}, Applied Mathematics Letters, 13 (2000), pp.~91--95.

\bibitem{Boros2019}
{\sc B.~Boros}, {\em Existence of positive steady states for weakly reversible
  mass-action systems}, SIAM Journal on Mathematical Analysis, 51 (2019),
  pp.~435--449.

\bibitem{NuclProp-DNA_book}
{\sc C.~R. Cantor and P.~R. Schimmel}, {\em Biophysical Chemistry: Part {III}.
  {T}he Behavior of Biological Macromolecules}, W.H. Freeman and Company, 1980.

\bibitem{NuclProp-polymer}
{\sc C.-S. Chern}, {\em Emulsion polymerization mechanisms and kinetics},
  Progress in Polymer Science, 31 (2006), pp.~443--486.

\bibitem{ConradiFeliuMinchevaWiuf2017}
{\sc C.~Conradi, E.~Feliu, M.~Mincheva, and C.~Wiuf}, {\em Identifying
  parameter regions for multistationarity}, PLOS Computational Biology, 13
  (2017), pp.~1--25.

\bibitem{CraciunFeinberg2005}
{\sc G.~Craciun and M.~Feinberg}, {\em Multiple equilibria in complex chemical
  reaction networks: {I}. {T}he injectivity property}, SIAM Journal on Applied
  Mathematics, 65 (2005), pp.~1526--1546.

\bibitem{CraciunFeinberg2006}
\leavevmode\vrule height 2pt depth -1.6pt width 23pt, {\em Multiple equilibria
  in complex chemical reaction networks: {II}. {T}he species-reaction graph},
  SIAM Journal on Applied Mathematics, 66 (2006), pp.~1321--1338.

\bibitem{CST}
{\sc G.~Craciun, B.~Joshi, C.~Pantea, and I.~Tan}, {\em Multistationarity of
  cyclic sequestration-transmutation networks}.
\newblock submitted.

\bibitem{CraciunMinchevaPanteaYu2020}
{\sc G.~Craciun, M.~Mincheva, C.~Pantea, and P.~Y. Yu}, {\em Delay stability of
  reaction systems}, Mathematical Biosciences, 326 (2020), p.~108387.

\bibitem{CraciunNazarovPantea2013_GAC}
{\sc G.~Craciun, F.~Nazarov, and C.~Pantea}, {\em Persistence and permanence of
  mass-action and power-law dynamical systems}, SIAM Journal on Applied
  Mathematics, 73 (2013), pp.~305--329.

\bibitem{craciun2011graph}
{\sc G.~Craciun, C.~Pantea, and E.~D. Sontag}, {\em Graph-theoretic analysis of
  multistability and monotonicity for biochemical reaction networks}, in Design
  and analysis of biomolecular circuits, Springer, 2011, pp.~63--72.

\bibitem{Epstein1990}
{\sc I.~R. Epstein}, {\em Differential delay equations in chemical kinetics:
  {S}ome simple linear model systems}, Journal of Chemical Physics, 92 (1990),
  pp.~1702--1712.

\bibitem{Feinberg_lecture}
{\sc M.~Feinberg}, {\em {Lectures On Chemical Reaction Networks}},  (1979).
\newblock Available at \url{https://crnt.osu.edu/LecturesOnReactionNetworks}.

\bibitem{Feinberg1987}
\leavevmode\vrule height 2pt depth -1.6pt width 23pt, {\em Chemical reaction
  network structure and the stability of complex isothermal reactors--{I}. the
  deficiency zero and deficiency one theorems}, Chemical Engineering Science,
  42 (1987), pp.~2229--2268.

\bibitem{FiedlerPtak1962}
{\sc M.~Fiedler and V.~Pt\'ak}, {\em On matrices with non-positive off-diagonal
  elements and positive principal minors}, Czechoslovak Mathematical Journal,
  12 (1962), pp.~382--400.

\bibitem{GopalkrishnanMillerShiu2014_GAC}
{\sc M.~Gopalkrishnan, E.~Miller, and A.~Shiu}, {\em A geometric approach to
  the global attractor conjecture}, SIAM Journal on Applied Dynamical Systems,
  13 (2014), pp.~758--797.

\bibitem{GraphHom}
{\sc P.~Hell and J.~Ne\u{s}et\u{r}il}, {\em Graphs and Homomorphisms}, vol.~28
  of Oxford Lecture Series in Mathematics and its Applications, Oxford
  University Press, 2004.

\bibitem{HofbauerSo2000}
{\sc J.~Hofbauer and J.~W.-H. So}, {\em Diagonal dominance and harmless
  off-diagonal delays}, Proceedings of the American Mathematical Society, 128
  (2000), pp.~2675--2682.

\bibitem{Johnson1974}
{\sc C.~R. Johnson}, {\em Second, third, and fourth order {$D$}-stability},
  Journal of Research of the National Bureau of Standards, Section B:
  Mathematical Sciences, 78B (1974), pp.~11--13.

\bibitem{LiptakHangosPitukSzederkenyi2018}
{\sc G.~Lipt\'ak, K.~M. Hangos, M.~Pituk, and G.~Szederk\'enyi}, {\em
  Semistability of complex balanced kinetic systems with arbitrary time
  delays}, Systems \& Control Letters, 114 (2018), pp.~38--43.

\bibitem{Macdonald1989}
{\sc N.~MacDonald}, {\em Biological delay systems: Linear stability theory},
  Cambridge University Press, 1989.

\bibitem{NuclProp-DNA}
{\sc M.~Manghi and N.~Destainville}, {\em Physics of base-pairing dynamics in
  {DNA}}, Physics Reports, 631 (2016), pp.~1--41.

\bibitem{MinchevaRoussel2007}
{\sc M.~Mincheva and M.~R. Roussel}, {\em Graph-theoretic methods for the
  analysis of chemical and biochemical networks. {II}. {O}scillations in
  networks with delays}, Journal of Mathematical Biology, 55 (2007),
  pp.~61--86.

\bibitem{NuclProp-DNA3}
{\sc S.~Mohan, C.~Hsiao, H.~VanDeusen, R.~Gallagher, E.~Krohn, B.~Kalahar,
  R.~M. Wartell, and L.~D. Williams}, {\em Mechanism of {RNA} double
  helix-propagation at atomic resolution}, The Journal of Physical Chemistry B,
  113 (2009), pp.~2614--2623.

\bibitem{NuclProp-micelle}
{\sc A.~L. Moore}, {\em 4 -- {P}roduction of Fluoroelastomers}, William Andrew
  Publishing, 2006, pp.~37--76.

\bibitem{NuclProp-DNA2}
{\sc B.~Rauzan, E.~McMichael, R.~Cave, L.~R. Sevcik, K.~Ostrosky, E.~Whitman,
  R.~Stegemann, A.~L. Sinclair, M.~J. Serra, and A.~A. Deckert}, {\em Kinetics
  and thermodynamics of dna, rna, and hybrid duplex formation}, Biochemistry,
  52 (2013), pp.~765--772.

\bibitem{RousselRoussel2001}
{\sc C.~J. Roussel and M.~R. Roussel}, {\em Delay-differential equations and
  the model equivalence problem in chemical kinetics}, Physics in Canada, 57
  (2001), pp.~114--120.

\bibitem{Roussel1996}
{\sc M.~R. Roussel}, {\em The use of delay differential equations in chemical
  kinetics}, Journal of Physical Chemistry, 100 (1996), pp.~8323--8330.

\bibitem{ShinarFeinberg2010}
{\sc G.~Shinar and M.~Feinberg}, {\em Structural sources of robustness in
  biochemical reaction networks}, Science, 327 (2010), pp.~1389--1391.

\bibitem{CraciunYu2018_review}
{\sc P.~Y. Yu and G.~Craciun}, {\em Mathematical analysis of chemical reaction
  systems}, Israel Journal of Chemistry, 58 (2018), pp.~733--741.

\end{thebibliography}

\end{document}